\documentclass[]{amsart}

\usepackage{amsmath}
\usepackage{amsfonts}
\usepackage{amsthm}
\usepackage{empheq}
\usepackage{array}
\usepackage{multirow}
\usepackage{color}
\usepackage{tikz}
\usepackage{graphicx}
\usepackage{natbib}
\usetikzlibrary{arrows}
\RequirePackage[colorlinks,citecolor=blue,urlcolor=blue]{hyperref}

\usepackage{amsaddr}
\usepackage{epstopdf}

\newtheorem{mytheorem}{Theorem}[subsection]
\newtheorem{mydefinition}{Definition}[subsection]
\newtheorem{mycorollary}{Corollary}[subsection]
\newtheorem{myremark}{Remark}[subsection]
\newtheorem{myassumption}{Assumption}
\newtheorem{mylemma}{Lemma}[subsection]
\newtheorem{aplemma}{Lemma}[section]

\newcolumntype{R}[1]{>{\raggedleft\let\newline\\\arraybackslash\hspace{0pt}}m{#1}}

\newcommand\numberthis{\addtocounter{equation}{1}\tag{\theequation}}

\begin{document}

%\begin{frontmatter}

\title[Third and Fourth Cumulants in ICA]{Joint Use of Third and Fourth Cumulants in Independent Component Analysis}
%\runtitle{Third and Fourth Cumulants in ICA}
\author{J. Virta, K. Nordhausen \and H. Oja}
\address{Department of Mathematics and Statistics \\ University of Turku, FIN-20014, Finland}

%\affil{Department of Mathematics and Statistics, University of Turku, FIN-20014, Finland}

%\begin{aug}
%\author{\fnms{Joni} \snm{Virta}\corref{}\ead[label=e1]{joni.virta@utu.fi}},
%\author{\fnms{Klaus} \snm{Nordhausen}\ead[label=e2]{klaus.nordhausen@utu.fi}}
%\and
%\author{\fnms{Hannu} \snm{Oja}
%\ead[label=e3]{hannu.oja@utu.fi}}

%\runauthor{J. Virta et al.}

%\affiliation{Department of Mathematics and Statistics, University of Turku, FIN-20014, Finland}

%\address{Addresses of authors\\
%\printead{e1}\\
%\printead{e2}\\
%\printead{e3}}
%\end{aug}
\maketitle

\begin{abstract}
The independent component model is  a latent variable model where the components of the observed random vector are linear combinations of latent independent variables. The aim is to find an estimate for a transformation matrix back to independent components. In moment-based approaches third cumulants are often  neglected in favour of fourth cumulants, even though both approaches have similar appealing properties. This paper considers the joint use of third and fourth cumulants in finding independent components. First, univariate cumulants are used as projection indices in search for independent components (projection pursuit). Second, multivariate  cumulant matrices are jointly used to solve the problem. The properties of the estimates are considered in detail through corresponding optimization problems,  estimating equations, algorithms  and asymptotic statistical properties. Comparisons of the asymptotic variances of different estimates in wide independent component models show that in most cases symmetric projection pursuit approach using both third  and fourth squared cumulants is a safe choice.
\end{abstract}

%, e.g. in the form of \textit{fourth order blind identification} (FOBI) or \textit{joint approximate diagonalization of eigen-matrices} (JADE).
%\begin{keyword}[class=MSC]
%\kwd[Primary ]{62H99}
%\kwd{60K35}
%\kwd[; secondary ]{60K35}
%\end{keyword}

%\begin{keyword}
%\kwd{Skewness}
%\kwd{Kurtosis}
%\kwd{FastICA}
%\kwd{FOBI}
%\kwd{JADE}
%\end{keyword}

%\end{frontmatter}

\section{Introduction}

In this paper we consider the use of third and fourth cumulants in independent component analysis (ICA).
The basic {\it blind source separation  model} assumes that the observed vectors $\textbf{x}_i \in \mathbb{R}^p$ are  linear combinations of some latent unobservable  variables $\textbf{z}_i \in \mathbb{R}^p$, $i=1,...,n$, the recovering of which is the objective of the analysis. If we write
\[
\textbf{X}=(\textbf{x}_1,...,\textbf{x}_n)^T\in \mathbb{R}^{n \times p}\ \ \mbox{and}\ \ \textbf{Z}=( \textbf{z}_1,...,\textbf{z} _n)^T\in \mathbb{R}^{n \times p}
\]
we have a semiparametric model
\[
\textbf{X}  =\mathbf{1}_n \boldsymbol{\mu}^T+ \textbf{Z} \mathbf{\Omega}^T
\]
with a shift vector $\boldsymbol{\mu} $ and  a non-singular transformation matrix or {\it mixing matrix} $\boldsymbol{\Omega} \in \mathbb{R}^{p \times p}$.
In the independent component model the columns of $\textbf{Z} $ are assumed to be independent, and in the most classical model it is further assumed that the rows
of $\textbf{Z} $, that is, $\textbf{z}_1,...,\textbf{z} _n$ are independent and identically distributed, each having $p$ independent components.
The model has the intuitive interpretation of $p$ hidden independent signals, the properties of which we observe only through an unknown linear mixing process. %As such it is commonly used in applications where latent sources can be thought to generate the observations e.g. in signal processing, brain imaging or analysis of financial time series.

{\it Projection pursuit (PP)} is  a popular method to reveal hidden structures in the data
by searching for low-dimensional orthogonal projections of interest. This is done by finding one or several linear combinations of the original variables
that maximize the value of an objective function, the so-called {\it projection index}. The classical measures of skewness and kurtosis, the third and fourth moments of a random variable after standardization, have been widely used for this purpose. \cite {Huber:1985} considered projection indices
with heuristic arguments that non-gaussian linear combinations are most interesting. His indices were ratios of two dispersion functionals thus  measuring kurtosis, with the classical kurtosis measure as a special case.  \cite{pena2001cluster} used projection pursuit for hidden cluster identification, again with the classical kurtosis measure. For early contributions on projection pursuit, see also \cite{FriedmanTukey:1974} and  \cite{jones1987projection}.  In the engineering literature, \cite{HyvarinenOja:1997} were the first to propose a projection pursuit approach  for independent component analysis with  the absolute value of the excess kurtosis, the fourth cumulant of a standardized variable, as a projection index and  considered later an extension with a choice among several alternative measures of non-gaussianity including the absolute value of the classical skewness, namely, the third cumulant of a standardized random variable. The approach is called {\it deflation-based  FastICA} or {\it symmetric FastICA} depending on whether the independent components are found one-by-one or simultaneously.  FastICA  is perhaps the most popular approach for the ICA problem in engineering applications. Recently, \cite{miettinen2014fourth} surveyed and discussed in detail the statistical properties of unmixing matrix estimates based on the use of the absolute value of the excess kurtosis as a projection index.

The concept and measures of kurtosis have been extended to the multivariate case as well. The
classical skewness and kurtosis measures by~\citet{Mardia:1970}, for example, combine in a natural
way the third and fourth moments of a standardized multivariate variable. Mardia's measures are
invariant under affine transformations. For other combinations of standardized third and fourth
moments, see also \cite{MoriRohtaguSzekely:1994,Kollo:2008}. In the invariant coordinate selection (ICS)
\citep{TylerCritchleyDumbgenOja:2009} one finds, using two scatter matrices, an unmixing matrix such that
the back-transformed variables are  presented in an invariant coordinate system, standardized
and ordered according their (generalized) kurtosis.   In independent component
analysis, certain scatter matrices based on fourth moments and the covariance matrix are used together in a similar
way to find the transformations to independent components; e.g. {\it fourth order blind identification (FOBI)} by~\citet{cardoso1989source} and
{\it joint approximate diagonalization of eigen-matrices (JADE)} by~\citet{cardoso1993blind} are regularly used in independent component analysis.
\cite{miettinen2014fourth} give a detailed survey of FOBI and JADE estimates with a comparison to deflation-based and symmetric projection pursuit estimates that use the absolute value of the excess kurtosis as a projection index. \cite{PenaPrietoViladomat:2010} use a fourth moment kurtosis matrix to reveal  cluster structures in the data. Similarly \cite{Loperfido:2013,Loperfido:2015} apply multivariate skewness measures for this purpose.

Independent component analysis has so far been mainly developed in the engineering literature and seen as a computational tool
to decompose a multivariate signal into independent non-gaussian signals. The procedures are then considered as numerical algorithms
rather than estimates of certain population quantities and considering their statistical properties has been neglected.  Recently, statisticians have
become interested in the problem. \citet{ChenBickel:2006} and \citet{SamworthYuan:2012} for example developed estimates that need only the existence of first moments and  rely on efficient nonparametric estimates of the marginal densities. Efficient estimation methods based on residual signed ranks and residual ranks have been developed recently by \citet{IlmonenPaindaveine:2011} and
\citet{HallinMehta:2015}.  %For a parametric model with marginal Pearson system approach, see \citet{KarvanenKoivunen:2002}.

%%%%%%%%%%%%%%%%%%%%%%%%%%%%%%%%%%%%%%

As far as the authors know, this paper introduces for the first time several ICA procedures that jointly use third and fourth cumulants. Only in the case of the JADE-type approach of Section \ref{subsec:jade} has this been done before, see \citet{moreau2001generalization}. First, weighted sums of squared third and fourth cumulants are used as projection indices in search for independent components (deflation-based and symmetric PP).  In most cases, our estimates then outperform the classical FastICA estimates that  use either absolute values of the third cumulants or absolute values of the fourth cumulants. Second, multivariate third and fourth cumulant matrices are jointly used to find an unmixing matrix estimate. Our approach is again novel in the sense that it uses also the multivariate third cumulant matrices. The classical FOBI and JADE estimates are found as special cases. The properties of the estimates are considered in detail through corresponding optimization problems,  estimating equations, algorithms  and asymptotic statistical properties. Comparisons of the asymptotic variances of different estimates in wide independent component models with skew, heavy- and light-tailed marginal distributions show that in most cases symmetric projection pursuit approach using third cumulants only outperforms its competitors.

The paper is structured as follows. We first introduce some helpful notation in  Section \ref{sec:nota}. After introducing the independent component (IC) model with relevant assumptions in Section \ref{sec:ICmodel}, the unmixing matrix estimates based on the projection pursuit approach  and those based on the multivariate cumulant matrices are discussed in detail in Sections  \ref{sec:uni} and \ref{sec:multi}, respectively. In Section \ref{sec:simu} the procedures are first compared in the case of cluster identification (using only one independent component)
 and then in the general case of $p$ independent components. %Due to affine equivariance of the estimates, it is then sufficient just to compare the asymptotic variances of the off-diagonal elements of the unmixing matrix estimates in the case $\mathbf{\Omega}=\textbf{I}_p$.
We end with some discussion on the results and their importance in Section \ref{sec:conc}. The proofs are reserved for the Appendix.

\section{Notation} \label{sec:nota}

For a univariate random variable $x$, we write $x_{st}=(x-E(x))/\sqrt{Var(x)}$ for its standardized version. The classical skewness, kurtosis and excess kurtosis of $x$ are then
\[
\gamma(x)=E\left(x_{st}^3\right),\ \ \beta(x)=E\left(x_{st}^4\right)\ \ \mbox{and}\ \
\kappa(x)=\beta(x)-3.
\]
Note that the measures $\gamma(x)$ and $\kappa(x)$ are the third and fourth cumulants of the standardized variable $x_{st}$. For symmetrical random variables $\gamma(x)=0$ and
for the normal distribution $\kappa(x)=0$.

 Throughout the paper we assume that $ \textbf{z}_1,...,\textbf{z}_n$ is a random sample from a $p$-variate distribution of $\textbf{z}$  with $E(\textbf{z})= \mathbf{0}$ and $Cov(\textbf{z})= \textbf{I}_p$ and that the $p$ components of $\textbf{z}$ are mutually independent.
 As different moment-based quantities play a crucial role in our derivations, we have the  shorthands %(and assumptions, see Assumption \ref{assu:fix})
 %for the raw moments of the observations generated by the independent component model discussed in Section \ref{sec:ICmodel}
\[  E(z_{ik}^3) =: \gamma_k, \quad E(z_{ik}^4) =: \beta_k  \quad \mbox{and}\  E(z_{ik}^4) - 3 =: \kappa_k,\ \ k=1,...,p.\]
For all $k=1,...,p$, the  moment-based expressions
\[E(z_{ik}^6) - E(z_{ik}^3)^2 =: \omega_k, \quad E(z_{ik}^4) - 1 =: \nu_k \quad \text{and} \quad E(z_{ik}^5) - E(z_{ik}^3) =: \eta_k \]
are encountered numerous times in the expressions for the asymptotic variances of our estimates and thus deserve symbols of their own.
The limiting distributions of our unmixing matrix estimates depend on the joint limiting distributions of
%The following four estimators are the building blocks of the asymptotic expressions and the covariances of all possible pairwise combinations of them are further listed in Table \ref{tab:cov}.
\begin{alignat*}{3}
\sqrt{n} \hat{s}_{kl} &= \dfrac{1}{\sqrt{n}}\sum_{i=1}^n z_{ik} z_{il}, \qquad & \\
\sqrt{n} \hat{r}_{kl} &= \dfrac{1}{\sqrt{n}}\sum_{i=1}^n (z_{ik}^2 - 1) z_{il}, \qquad & \sqrt{n} \hat{r}_{mkl} &= \frac{1}{\sqrt{n}} \sum_{i=1}^n z_{im} z_{ik} z_{il}, \\
\sqrt{n} \hat{q}_{kl} &= \frac{1}{\sqrt{n}} \sum_{i=1}^n (z_{ik}^3 - \gamma_k) z_{il} \qquad \mbox{and} \qquad & \sqrt{n} \hat{q}_{mkl} &= \frac{1}{\sqrt{n}} \sum_{i=1}^n z_{im}^2 z_{ik} z_{il}.
\end{alignat*}
Central limit theorem can be used to prove the joint limiting multinormality of these statistics with the  variances and covariances as listed in Table \ref{tab:cov}.

%\sqrt{n} \hat{t}_{kl} &= \dfrac{1}{\sqrt{n}}\sum_{i=1}^n z_{ik}^2 z_{il}

\begin{table}[h]
\caption{Covariances of the column and row entries, for $k \neq l \neq m \neq m'$.}
\label{tab:cov}
\begin{center}
\begin{tabular}{| r || R{1cm} | R{1cm} | R{1cm} | R{1cm} |  R{1cm} | R{1cm} | R{1cm} | R{1cm} | R{1cm} |}
\hline
 & $\sqrt{n}\hat{q}_{kl}$ & $\sqrt{n}\hat{q}_{lk}$ & $\sqrt{n}\hat{r}_{kl}$ & $\sqrt{n}\hat{r}_{lk}$ & $\sqrt{n} \hat{q}_{m'kl}$ & $\sqrt{n} \hat{r}_{mkl}$ & $\sqrt{n}\hat{s}_{kl}$ \\ \hline \hline
$\sqrt{n}\hat{q}_{kl}$ & $\omega_k$ & $\beta_k \beta_l$ & $\eta_k$ & $\beta_k \gamma_l$ & $\beta_k$ & 0 & $\beta_k$ \\ \hline
$\sqrt{n}\hat{q}_{lk}$ &  & $\omega_l$ & $\beta_l \gamma_k$ & $\eta_l$ & $\beta_l$ & 0 & $\beta_l$ \\ \hline
%$\hat{t}_{kl}$ &  &  & $\beta_k$ & $\gamma_k \gamma_l$ & & & 0 & 1 & $\gamma_k$ \\ \hline
%$\hat{t}_{lk}$ &  &  &  & $\beta_l$ & & & 1 & 0 & $\gamma_l$ \\ \hline
$\sqrt{n}\hat{r}_{kl}$ & & & $\nu_k$ & $\gamma_k \gamma_l$ & $\gamma_k$ & 0 & $\gamma_k$ \\ \hline
$\sqrt{n}\hat{r}_{lk}$ & & & & $\nu_l$ & $\gamma_l$ & 0 & $\gamma_l$ \\ \hline
%$\bar{z}_k$ &  &  &  & & & & 1 & 0 & 0 \\ \hline
%$\bar{z}_k$ &  &  &  & & & &  & 1 & 0 \\ \hline
$\sqrt{n}\hat{q}_{m'kl}$ &  & & & & $\beta_m$ & 0 & 1 \\ \hline
$\sqrt{n}\hat{r}_{mkl}$ &  & & & & & 1 & 0  \\ \hline
$\sqrt{n}\hat{s}_{kl}$ &  & & & & & & 1 \\
\hline
\end{tabular}
\end{center}
\end{table}

For a $p$-variate random vector $\textbf{x}$ with mean vector $\boldsymbol{\mu}$ and covariance matrix $\boldsymbol{\Sigma}$, the standardized vector is  $\textbf{x}_{st} = \boldsymbol{\Sigma}^{-1/2} (\textbf{x} - \boldsymbol{\mu})$, where $\boldsymbol{\Sigma}^{-1/2}$ is chosen as the  symmetric matrix $\textbf{G}$ satisfying $\textbf{G} \boldsymbol{\Sigma} \textbf{G} = \textbf{I}_p$. A useful result (see for example \cite{ilmonen2012invariant}) regarding the standardized observations is that if $\textbf{x}^* = \textbf{A} \textbf{x} + \textbf{b}$, then $\textbf{x}^*_{st} = \textbf{U} \textbf{x}_{st}$, for some orthogonal matrix $\textbf{U} \in \mathbb{R}^{p \times p}$. This fact is used in proving the affine equivariances of the different functionals later on. Additionally, the centered observations are in the proofs denoted with $\tilde{\textbf{z}}_i := \textbf{z}_i - \bar{\textbf{z}}$ for clarity.

The standard basis vectors of $\mathbb{R}^p$ are denoted by $\textbf{e}_i \in \mathbb{R}^p$. That is, the $j$th element of $\textbf{e}_i$ is equal to Kronecker's delta $\delta_{ij} = I(i = j)$. Using the standard basis vectors we further define the following matrices
\[\textbf{E}^{ij} = \textbf{e}_i \textbf{e}_j^T, \qquad i,j=1,...,p,\]
the only non-zero element of $\textbf{E}^{ij}$ being the element $(i, j)$. Finally, some often encountered sets of matrices are denoted with symbols of their own:
\begin{itemize}
\item $\mathcal{U} = \{\textbf{U} \in \mathbb{R}^{p \times p} \,: \, \textbf{U} \text{ is an orthogonal matrix.}\}$
\item $\mathcal{J} = \{\textbf{J} \in \mathbb{R}^{p \times p}\,: \, \textbf{J} = diag(j_1,...,j_p), \, j_1,...,j_p = \pm 1\}$
\item $\mathcal{D} = \{\textbf{D} \in \mathbb{R}^{p \times p}\,: \, \textbf{D} = diag(d_1,...,d_p), \, d_1,...,d_p > 0\}$
\item $\mathcal{P} = \{\textbf{P} \in \mathbb{R}^{p \times p}\,: \, \textbf{P} \text{ is a permutation matrix.}\}$
\end{itemize}

\section{Independent component model}\label{sec:ICmodel}

The model used throughout the paper is the independent component model (IC model), in which the $p$-variate observations $\textbf{x}_1,..., \textbf{x}_n$ are thought to originate as
\begin{align}\label{eq:icm_icmodel}
\textbf{x}_i = \boldsymbol{\mu} + \boldsymbol{\Omega} \textbf{z}_i, \qquad i=1,...,n,
\end{align}
where the unobserved, independent and identically distributed vectors $\textbf{z}_i=(z_{i1},...,z_{ip})^T$ satisfy the following three assumptions.

\begin{myassumption}\label{assu:ind}
$z_{i1},...,z_{ip}$  are standardized and mutually independent.
\end{myassumption}

%\begin{myassumption}\label{assu:fix}
%$E(z_{ik})=0$, $E(z_{ik}^2)=1$ and $E(z_{ik}^4)$ exist, $k=1,...,p$.
%\end{myassumption}

\begin{myassumption}\label{assu:gaus}
At most one of $z_{i1},...,z_{ip}$ is normally distributed.
\end{myassumption}
%(Note that in the following also non-subscripted $\textbf{x}$ and $\textbf{z}$ are used when not referring to any specific observation.)
 The conditions $E(z_{ik})=0$ and $E(z_{ik}^2)=1$, $k=1,...,p$, in Assumption  \ref{assu:ind} just serve as identification constraints for the location $\boldsymbol{\mu}$ and the lengths of the rows of $\mathbf{\Omega}$.
  Then $$E(\textbf{x}_i)=\boldsymbol{\mu}\ \ \mbox{and}\ \ Cov(\textbf{x}_i)=\boldsymbol{\Sigma} = \boldsymbol{\Omega} \boldsymbol{\Omega}^T.$$  To see why Assumption \ref{assu:gaus} has to hold, consider the case $\textbf{z}_i \sim \mathcal{N}_2(\textbf{0}, \textbf{I}_p)$. Then any orthogonal transformation preserves the distribution of $\textbf{z}_i$, that is, $\textbf{z}_i \sim \textbf{U} \textbf{z}_i$ for all $\textbf{U} \in \mathcal{U}$, and we can recover the original $\textbf{z}_i$ only up to some orthogonal matrix $\textbf{U}$. Regarding the uniqueness of the independent components
after our assumptions, it is easy to see that the signs and the order of the independent components are not fixed  in the model.  This, however, is satisfactory in most applications.

Additionally, we introduce the following six assumptions, each of which is a stricter version of Assumption \ref{assu:gaus} and implicitly assumes that the third and fourth moments exist.
%We need the assumption of the existence of the fourth moments as our procedures are based on the first four moments; this assumption of course rules out heavy-tailed distributions.
This hence rules out heavy-tailed distributions.
The relevance of these assumptions will become apparent in later discussions on the existence and properties of different unmixing matrix functionals. Recall that
\[
\gamma_k=\gamma (z_{ik})=E(z_{ik}^3) \ \ \mbox{and} \ \ \kappa_k=\kappa(z_{ik})=E(z_{ik}^4)-3 ,\ \ k=1,...,p.
\]

\begin{myassumption}\label{assu:zeroskew}
At most one of $\gamma_1,...,\gamma_p$ is zero.
%At most one of the independent components has skewness $\gamma_k$ of zero.
\end{myassumption}

\begin{myassumption}\label{assu:zerokurt}
At most one of $\kappa_1,...,\kappa_p$ is zero.
%At most one of the independent components has excess kurtosis  $\kappa_k$ of zero.
\end{myassumption}

\begin{myassumption}\label{assu:distskew}
$\gamma_1,...,\gamma_p$ are distinct.
%The skewnesses $\gamma_k$ of the independent components are distinct.
\end{myassumption}

\begin{myassumption}\label{assu:distkurt}
$\kappa_1,...,\kappa_p$ are distinct.
%The excess kurtoses $\kappa_k$ of the independent components are distinct.
\end{myassumption}

\begin{myassumption}\label{assu:zeroboth}
For at most one $k$, $\gamma_k=\kappa_k=0$.
%At most one of the independent components has \textbf{both} zero skewness $\gamma_k$ \textbf{and} zero excess kurtosis $\kappa_k$.
\end{myassumption}

\begin{myassumption}\label{assu:distboth}
There is no  $k\ne l$ such that $\gamma_k=\gamma_l$ and $\kappa_k=\kappa_l$.
%No two of the independent components have \textbf{both} same skewness $\gamma_k$ \textbf{and} same excess kurtosis $\kappa_k$.
\end{myassumption}

Assumption \ref{assu:zeroskew} is often considered to be much more restrictive
than Assumption \ref{assu:zerokurt} as it limits the number of symmetric sources to one. %Note that assumptions regarding the symmetry of the independent components have been considered also before in the literature; e.g. in \cite{MiettinenNordhausenOjaTaskinen2014} our Assumption \ref{assu:zeroskew} is evoked when using the non-linearity $g(x) = x^2$ and conversely in
The assumption of symmetric sources is made in \cite{IlmonenPaindaveine:2011}. Their approach allows however heavy-tailed distributions as the existence of moments is not assumed. Note also that Assumptions \ref{assu:distskew}, \ref{assu:distkurt} and \ref{assu:distboth} rule out components with identical marginal distributions.

The structure of the  assumptions is depicted in Figure~\ref{fig:icm_assum}. From the graph we again see that all the ``moment-based assumptions'' are stronger than  Assumption \ref{assu:gaus} and  the most stringent amongst them are Assumptions \ref{assu:distskew} and \ref{assu:distkurt}.

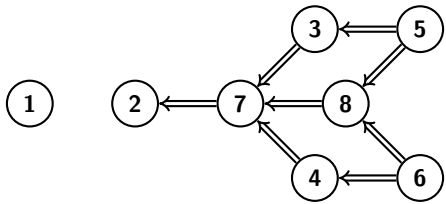
\begin{figure}[h]
\begin{center}
\begin{tikzpicture}[auto,node distance=1.4cm,
  thick,main node/.style={circle,draw,font=\sffamily\small\bfseries}]
\node[main node] (1) {1};
\node[main node] (2) [right of=1] {2};
\node[main node] (7) [right of=2] {7};
\node[main node] (3) [above right of=7] {3};
\node[main node] (4) [below right of=7] {4};
\node[main node] (5) [right of=3] {5};
\node[main node] (6) [right of=4] {6};
\node[main node] (8) [below left of=5] {8};

\draw[double,double equal sign distance,-implies] (3) to (7);
\draw[double,double equal sign distance,-implies] (4) to (7);
\draw[double,double equal sign distance,-implies] (7) to (2);
\draw[double,double equal sign distance,-implies] (8) to (7);
\draw[double,double equal sign distance,-implies] (5) to (3);
\draw[double,double equal sign distance,-implies] (6) to (4);
\draw[double,double equal sign distance,-implies] (5) to (8);
\draw[double,double equal sign distance,-implies] (6) to (8);
\end{tikzpicture}

\caption{The relationships and implications between the different assumptions on the independent components.}
\label{fig:icm_assum}
\end{center}
\end{figure}

Next we state one of the key results of independent component analysis, the proof of which can be found, e.g., in \cite{miettinen2014fourth}.

\begin{mytheorem}\label{theo:icm_rotate}
Let $\textbf{x} \in \mathbb{R}^p$ follow the independent component model in \eqref{eq:icm_icmodel}. Then the standardized observations $\textbf{x}_{st} = \boldsymbol{\Sigma}^{-1/2} (\textbf{x} - \boldsymbol{\mu})$ %, where $\boldsymbol{\Sigma} = \boldsymbol{\Omega} \boldsymbol{\Omega}^T$,
satisfy $\textbf{z} = \textbf{U} \textbf{x}_{st}$ for some orthogonal matrix $\textbf{U}$.
\end{mytheorem}
Theorem \ref{theo:icm_rotate} essentially states that the estimation of the unmixing matrix $\boldsymbol{\Omega}^{-1}$ can in fact be reduced to a simpler task, namely to the estimation of an orthogonal  matrix $\textbf{U}$.  This result is used repeatedly in the following sections.

Finally, we define the independent component functional $\textbf{W}(F)$ as follows.
\begin{mydefinition}
The functional $\textbf{W}(F) \in \mathbb{R}^{p \times p}$ is said to be an independent component functional if (i) $\textbf{W}(F_\textbf{x}) \textbf{x}$ has independent components under the independent component model \eqref{eq:icm_icmodel} and (ii) $\textbf{W}(F_\textbf{x})$ is affine equivariant in the sense that
for all $\textbf{x} $, all full-rank $\textbf{A} \in \mathbb{R}^{p \times p}$ and $\textbf{b} \in \mathbb{R}^p$, there exist $\textbf{P} \in \mathcal{P}$ and $\textbf{J} \in \mathcal{J}$
such that
\[\textbf{W}(F_{\textbf{Ax} + \textbf{b}}) \textbf{A} \textbf{x} = \textbf{P} \textbf{J} \textbf{W}(F_\textbf{x}) \textbf{x}.\]
\end{mydefinition}
Note that the functional $\textbf{W}(F)$ is defined at any $F$ and is required to be Fisher consistent to $\boldsymbol{\Omega}^{-1}$ up to permutation and heterogeneous sign-changes of the rows. The functional  $\textbf{W}(F)$   is  \textit{affine equivariant} and therefore provides a transformation  to an {\it invariant coordinate system (ICS)}, that is, it is also an {\it ICS functional};  see \cite{TylerCritchleyDumbgenOja:2009} and \cite{ilmonen2012invariant}.
Let next $F_n$ be the empirical cumulative distribution function from a random sample $\textbf{x}_1,...,\textbf{x}_n$ from $F$. Then $\textbf{W}(F_n)$ provides a natural affine equivariant estimate of $\textbf{W}(F)$. The affine equivariance property  simplifies the derivation of the asymptotic behavior of $\textbf{W}(F_n)$ considerably as we may restrict our attention to the case $\boldsymbol{\Omega} = \textbf{I}_p$ only.

Finally, note that Assumption \ref{assu:gaus} guarantees that the estimated vector of independent components is indeed equal to $\textbf{z}$ up to sign and order, that is, all independent component functionals $\textbf{W}(F_\textbf{x})$ lead to the same independent components up to sign change and permutation; see the Ghurye-Olkin-Zinger characterization theorem in \citet{ibragimov2014ghurye}.

\section{Univariate third and fourth cumulants}\label{sec:uni}

We first consider the use of univariate third and fourth cumulants in estimating the unmixing matrix, leading in old and new variants of the so-called deflation-based FastICA and symmetric FastICA.

\subsection{Estimating the components separately} \label{subsec:sep}

First, to actually guarantee the validity of our approach, we prove the following inequality, an extension of  Theorem 2 in \cite{miettinen2014fourth}.

\begin{mytheorem}\label{theo:sep_ineq}
Let $\textbf{z} \in \mathbb{R}^p$ have independent components with $E(\textbf{z}) = \textbf{0}$ and $Cov(\textbf{z}) = \textbf{I}_p$.  Then
\begin{align*}
\alpha_1 \gamma^2(\textbf{u}^T \textbf{z}) + \alpha_2 \kappa^2(\textbf{u}^T \textbf{z}) \leq \underset{1 \leq k \leq p}{\emph{max}} \left( \alpha_1 \gamma_k^2 + \alpha_2 \kappa_k^2 \right),
%\alpha_1 \left( E\left[\left(\textbf{u}^T \textbf{z}\right)^3 \right] \right)^2 + \alpha_2 \left( E\left[\left(\textbf{u}^T \textbf{z}\right)^4 - 3 \right] \right)^2 \leq \underset{1 \leq k %\leq p}{\emph{max}} \left( \alpha_1 \gamma_k^2 + \alpha_2 \kappa_k^2 \right),
\end{align*}
for all $\alpha_1, \alpha_2 \in \mathbb{R}^+ \cup \{0\}$ and for all vectors $\textbf{u} \in \mathbb{R}^p$ satisfying $\textbf{u}^T \textbf{u} = 1$.
\end{mytheorem}
The inequality in Theorem \ref{theo:sep_ineq} implies that the independent components can be recovered by repeatedly searching for mutually orthogonal vectors
 $\textbf{u}$ maximizing the projection index
 \[ \alpha_1 \gamma^2(\textbf{u}^T \textbf{x}_{st}) + \alpha_2 \kappa^2(\textbf{u}^T \textbf{x}_{st}) \]
 and we give the following.

\begin{mydefinition}\label{def:sep_func}
The deflation-based projection pursuit functional based on squared third and fourth cumulants is a functional  $\textbf{W} (F_\textbf{x}) = \textbf{U}\boldsymbol{\Sigma}^{-1/2}$, where  $\boldsymbol{\Sigma} = Cov(\textbf{x})$ and the rows of the orthogonal matrix $\textbf{U} = (\textbf{u}_1,...,\textbf{u}_p)^T$ are found one-by-one, such that
\[ \textbf{u}_k = \underset{\textbf{u}_k^T \textbf{u}_j = \delta_{kj}, 1 \leq j \leq k}{\emph{argmax}} \left( \alpha \gamma^2(\textbf{u}_k^T \textbf{x}_{st}) + (1 - \alpha)
\kappa^2(\textbf{u}_k^T \textbf{x}_{st}) \right), \]
where $\alpha \in [0, 1]$ is the proportion of weight given to third cumulants.
\end{mydefinition}

Note that weights $\alpha_1$ and $\alpha_2$ and weights $\alpha_1/(\alpha_1+\alpha_2)$ and $\alpha_2/(\alpha_1+\alpha_2)$ in Theorem \ref{theo:sep_ineq} lead to the same optimization problem, and we may without loss of generality use just a single weight parameter $\alpha=\alpha_1/(\alpha_1+\alpha_2)$. An interesting choice is $\alpha=0.8$ corresponding to
\[
\frac {\gamma^2(\textbf{u}^T \textbf{x}_{st})}6   + \frac { \kappa^2(\textbf{u}^T \textbf{x}_{st})} { 24}
\]
as we are then maximizing the value of a functional that is often used to test for univariate normality (see \citet{jarque1987test}).
 %Note also, that instead of using the weight $\alpha$, one could opt to take into consideration the differing scales of skewness and excess kurtosis by first weighting them with their asymptotic variances, 6 and 24 respectively, and then applying weights $\alpha_0$ and $1-\alpha_0$. The different weighting systems then have a relationship $\alpha = (4\alpha_0)/(3\alpha_0 + 1)$. The choice of weighting system naturally has no effect on the theoretical results and for simplicity we choose to use the former.
Note also, that choosing either $\alpha = 0$ or $\alpha = 1$ makes the proposed method equivalent to the so-called \textit{deflation-based FastICA} \citep{Hyvarinen:1999a} with the projection indices
$|\gamma(\textbf{u}_k^T \textbf{x}_{st})|$ and $|\kappa(\textbf{u}_k^T \textbf{x}_{st})|$, respectively. For general results concerning deflation-based FastICA using absolute values see also \cite{Ollila2010,nordhausen2011deflation,MiettinenNordhausenOjaTaskinen:2014}.

The affine equivariance of the procedure given in Definition \ref{def:sep_func} follows simply from the fact that the optimization problem along with the constraints is invariant under mappings $\textbf{x}_{st} \mapsto \textbf{V} \textbf{x}_{st}$, where $ \textbf{V} \in \mathcal{U}$. (Recall that the transformation $\textbf{x}\to \textbf{A} \textbf{x} + \textbf{b}$
induces the transformation $\textbf{x}_{st}\to  \textbf{V}  \textbf{x}_{st}$ for some orthogonal $\textbf{V}$.)
This together with Theorem \ref{theo:sep_ineq} implies the following.

\begin{mylemma}\label{lem:sep_ae}
The deflation-based projection pursuit functional  $\textbf{W} (F_\textbf{x})$ in Definition \ref{def:sep_func}  is an independent component functional for every $\alpha \in [0, 1]$.
\end{mylemma}

The Lagrangian of the maximization problem involving $\textbf{u}_k$ has the form
\begin{align*}
 L(\textbf{u}_k, \boldsymbol{\lambda}_k) &= \alpha \left( E\left[(\textbf{u}_k^T \textbf{x}_{st})^3 \right] \right)^2 + (1 - \alpha) \left( E\left[(\textbf{u}_k^T \textbf{x}_{st})^4 \right] - 3 \right)^2 \\
 &- \sum_{j=1}^{k-1} \lambda_{kj} \textbf{u}_j^T \textbf{u}_k - \lambda_{kk} (\textbf{u}_k^T \textbf{u}_k - 1).
\end{align*}
First differentiating w.r.t. $\textbf{u}_k$ and the Lagrangian multipliers and then solving for the Lagrangian multipliers and substituting them back in yields the following estimating equation for the $k$th row $\textbf{u}_k$.
\[
\left(\textbf{I}_p- \sum_{j=1}^k \textbf{u}_j \textbf{u}_j^T \right)  \textbf{T}_k=\textbf{0}
\]
where
\begin{eqnarray*}
\textbf{T}_k &=& 3 \alpha E\left[(\textbf{u}_k^T \textbf{x}_{st})^3\right] E\left[(\textbf{u}_k^T \textbf{x}_{st})^2 \textbf{x}_{st} \right]\\
             &+& 4 (1 - \alpha) \left(E\left[(\textbf{u}_k^T \textbf{x}_{st})^4 \right] - 3\right)  E\left[(\textbf{u}_k^T \textbf{x}_{st})^3 \textbf{x}_{st} \right].
\end{eqnarray*}
After finding $\textbf{u}_1,...,\textbf{u}_{k-1}$,  we then obtain a fixed-point solution for $\textbf{u}_k$  by successively iterating over the the following steps.
\begin{enumerate}
\item $\textbf{u}_k \leftarrow \left(\textbf{I}_p - \sum_{j=1}^{k-1} \textbf{u}_j \textbf{u}_j^T \right) \textbf{T}_k$.
\item $\textbf{u}_k \leftarrow \|\textbf{u}_k\|^{-1} \textbf{u}_k$.
\end{enumerate}
A Newton-Raphson type algorithm for this problem might be more efficient and will be considered in a separate paper.

Additionally, the estimating equations provide us with the following results regarding the asymptotic behavior of the unmixing matrix estimates $\hat{\textbf{W}}$ in the case $\boldsymbol{\Omega} = \textbf{I}_p$.  Note that, this is sufficient as all estimates are affine equivariant. The general case easily follows.

\begin{mytheorem}\label{theo:sep_asymp}
(i) Let $\textbf{z}_1,...,\textbf{z}_n$ be a random sample from a distribution with finite sixth moments and satisfying assumptions \ref{assu:ind} and \ref{assu:zeroskew}. Then there exists a sequence of solutions based on skewness (that is, $\alpha = 1$) such that $\hat{\textbf{W}}\rightarrow_P \textbf{I}_p$ and
\begin{align*}
\sqrt{n} \hat{w}_{kl} &= -\sqrt{n} \hat{w}_{lk} - \sqrt{n} \hat{s}_{kl} + o_P(1), \quad l < k, \\
\sqrt{n} (\hat{w}_{kk} - 1) &= -\dfrac{1}{2} \sqrt{n} (\hat{s}_{kk} - 1) + o_P(1), \\
\sqrt{n} \hat{w}_{kl} &= \frac{\sqrt{n} \hat{\psi}_{1kl}}{\gamma_k^2} + o_P(1), \quad l > k,
\end{align*}
where $\hat{\psi}_{1kl} = \gamma_k \hat{r}_{kl} - \gamma_k^2 \hat{s}_{kl}$.
\newline

(ii) Let $\textbf{z}_1,...,\textbf{z}_n$ be a random sample from a distribution with finite eighth moments and satisfying assumptions  \ref{assu:ind} and \ref{assu:zerokurt}. Then there exists a sequence of solutions based on kurtosis (that is, $\alpha = 0$) such that $\hat{\textbf{W}}\rightarrow_P \textbf{I}_p$ and
\begin{align*}
\sqrt{n} \hat{w}_{kl} &= -\sqrt{n} \hat{w}_{lk} - \sqrt{n} \hat{s}_{kl} + o_P(1), \quad l < k, \\
\sqrt{n} (\hat{w}_{kk} - 1) &= -\dfrac{1}{2} \sqrt{n} (\hat{s}_{kk} - 1) + o_P(1), \\
\sqrt{n} \hat{w}_{kl} &= \frac{\sqrt{n} \hat{\psi}_{2kl}}{\kappa_k^2} + o_P(1), \quad l > k,
\end{align*}
where $\hat{\psi}_{2kl} = \kappa_k \hat{q}_{kl} - \kappa_k \beta_k \hat{s}_{kl}$.
\newline

(iii) Let $\textbf{z}_1,...,\textbf{z}_n$ be a random sample from a distribution with finite eighth moments and satisfying assumptions \ref{assu:ind} and \ref{assu:zeroboth}. Then there exists a sequence of solutions based on both skewness and kurtosis such that $\hat{\textbf{W}}\rightarrow_P \textbf{I}_p$ and
\begin{align*}
\sqrt{n} \hat{w}_{kl} &= -\sqrt{n} \hat{w}_{lk} - \sqrt{n} \hat{s}_{kl} + o_P(1), \quad l < k, \\
\sqrt{n} (\hat{w}_{kk} - 1) &= -\dfrac{1}{2} \sqrt{n} (\hat{s}_{kk} - 1) + o_P(1), \\
\sqrt{n} \hat{w}_{kl} &= \frac{3 \alpha \sqrt{n} \hat{\psi}_{1kl} + 4 (1 - \alpha) \sqrt{n} \hat{\psi}_{2kl}}{3 \alpha \gamma_k^2 + 4 (1 - \alpha) \kappa_k^2} + o_P(1), \quad l > k,
\end{align*}
where $\alpha \in [0, 1]$ is the proportion of weight given to skewness, and $\hat{\psi}_{1kl}$ is as in (i) and $\hat{\psi}_{2kl}$ as in (ii).
\end{mytheorem}

\begin{mycorollary} \label{cor:sep_asymp}
(i) Under the assumptions of Theorem \ref{theo:sep_asymp}(i) the limiting distribution of $\sqrt{n} \, vec(\hat{\textbf{W}} - \textbf{I}_p)$ is multivariate normal with mean vector $\textbf{0}$ and elementwise variances
\begin{alignat*}{3}
&ASV&(\hat{w}_{kl}) &= \frac{\zeta_{11l}}{\gamma_l^4} + 1, &\quad l < k, \\
&ASV&(\hat{w}_{kk}) &= \frac{\kappa_k + 2}{4}, &\\
&ASV&(\hat{w}_{kl}) &= \frac{\zeta_{11k}}{\gamma_k^4}, &\quad l > k,
\end{alignat*}
where $\zeta_{11k} = \gamma_k^2 (\nu_k - \gamma_k^2)$.
\newline

(ii) Under the assumptions of Theorem \ref{theo:sep_asymp}(ii) the limiting distribution of $\sqrt{n} \, vec(\hat{\textbf{W}} - \textbf{I}_p)$ is multivariate normal with mean vector $\textbf{0}$ and elementwise variances
\begin{alignat*}{3}
&ASV&(\hat{w}_{kl}) &= \frac{\zeta_{22l}}{\kappa_l^4} + 1, &\quad l < k, \\
&ASV&(\hat{w}_{kk}) &= \frac{\kappa_k + 2}{4}, &\\
&ASV&(\hat{w}_{kl}) &= \frac{\zeta_{22k}}{\kappa_k^4}, &\quad l > k,
\end{alignat*}
where $\zeta_{22k} = \kappa_k^2 (\omega_k - \beta_k^2) $.
\newline

(iii) Under the assumptions of Theorem \ref{theo:sep_asymp}(iii) the limiting distribution of $\sqrt{n} \, vec(\hat{\textbf{W}} - \textbf{I}_p)$ is multivariate normal with mean vector $\textbf{0}$ and elementwise variances
\begin{alignat*}{3}
&ASV&(\hat{w}_{kl}) &= \frac{9 \alpha^2 \zeta_{11l} + 16 (1 - \alpha)^2 \zeta_{22l} + 24 \alpha (1 - \alpha) \zeta_{12l}}{(3 \alpha \gamma_l^2 + 4 (1 - \alpha) \kappa_l^2)^2} + 1, &\quad l < k, \\
&ASV&(\hat{w}_{kk}) &= \frac{\kappa_k + 2}{4}, &\\
&ASV&(\hat{w}_{kl}) &= \frac{9 \alpha^2 \zeta_{11k} + 16 (1 - \alpha)^2 \zeta_{22k} + 24 \alpha (1 - \alpha) \zeta_{12k}}{(3 \alpha \gamma_k^2 + 4 (1 - \alpha) \kappa_k^2)^2}, &\quad l > k,
\end{alignat*}
where $\alpha \in [0, 1]$ is the proportion of weight given to skewness, and $\zeta_{11k}$ is as in (i), $\zeta_{22k}$ as in (ii) and $\zeta_{12k} = \gamma_k \kappa_k (\eta_k - \gamma_k \beta_k)$.
\end{mycorollary}

\subsection{Estimating the components simultaneously}\label{symm FastIca}

As in Section \ref{subsec:sep}, we first provide the justification for the validity of our approach in the form of the following inequality.

\begin{mytheorem}\label{theo:sym_ineq}
Let $\textbf{z} \in \mathbb{R}^p$ have independent components with $E(\textbf{z}) = \textbf{0}$ and $Cov(\textbf{z}) = \textbf{I}_p$.  Then
\[
\alpha \sum_{k=1}^p \gamma^2(\textbf{u}_k^T \textbf{z}) + (1-\alpha) \sum_{k=1}^p \kappa^2(\textbf{u}_k^T \textbf{z})
\leq  \alpha \sum_{k=1}^p \gamma_k^2 + (1-\alpha) \sum_{k=1}^p \kappa_k^2,
\]
for all orthogonal matrices $\textbf{U} = (\textbf{u}_1,..., \textbf{u}_p)^T \in \mathbb{R}^{p \times p}$ and for all $\alpha \in [0, 1]$.
\end{mytheorem}
The inequality in Theorem \ref{theo:sym_ineq} suggests the following strategy for searching for the independent components.
\begin{mydefinition}\label{def:sym_func}
The symmetric projection pursuit functional based on squared third and fourth cumulants is a functional  $\textbf{W} (F_\textbf{x}) = \textbf{U}\boldsymbol{\Sigma}^{-1/2}$, where  $\boldsymbol{\Sigma} = Cov(\textbf{x})$ and the rows of the orthogonal matrix $\textbf{U} = (\textbf{u}_1,...,\textbf{u}_p)^T$ are found simultaneously, such that
\[ \textbf{U} = \underset{\textbf{U} \in \mathcal{U}}{\emph{argmax}} \left( \alpha \sum_{k=1}^p \gamma^2(\textbf{u}_k^T \textbf{x}_{st}) + (1 - \alpha) \sum_{k=1}^p \kappa^2(\textbf{u}_k^T \textbf{x}_{st}) \right), \]
where $\alpha \in [0, 1]$ is the proportion of weight given to third cumulants.
\end{mydefinition}

 Recall that in the classical {\it symmetric fastICA} approach utilizing third or fourth cumulants one finds $\textbf{U}$ that maximizes either  $ \sum_{k=1}^p |\gamma(\textbf{u}_k^T \textbf{x}_{st})|$ or $ \sum_{k=1}^p |\kappa(\textbf{u}_k^T \textbf{x}_{st})|$. We thus use squares instead of absolute values and both cumulants simultaneously. See also \cite{Wei2014,miettinen2014fourth} for more details on the approach using absolute values.

In \citet{comon1994independent} the projection indices that satisfy inequalities such as in Theorem \ref{theo:sym_ineq} are called  \textit{contrasts}, see also \citet{moreau2001generalization}.
Both papers also show that in general any cumulants of order 3 or higher can be used in independent component analysis as contrasts.

It is easy to see that the functional  in  Definition \ref{def:sym_func} is affine equivariant and  Theorem \ref{theo:sym_ineq}  implies the following.

\begin{mylemma}\label{lem:symm_ae}
The deflation-based projection pursuit functional  $\textbf{W} (F_\textbf{x})$ in Definition \ref{def:sym_func}  is an independent component functional for every $\alpha \in [0, 1]$.
\end{mylemma}

The Lagrangian of the maximization problem in Definition \ref{def:sym_func} has the form
\begin{align*}
 L(\textbf{U}, \boldsymbol{\Lambda}) &= \alpha \sum_{k=1}^p \left( E\left[(\textbf{u}_k^T \textbf{x}_{st})^3 \right] \right)^2 + (1 - \alpha) \sum_{k=1}^p \left( E\left[(\textbf{u}_k^T \textbf{x}_{st})^4 \right] - 3 \right)^2 \\
 &- \sum_{k=1}^{p-1} \sum_{l=k+1}^p \lambda_{kl} \textbf{u}_k^T \textbf{u}_l - \sum_{k=1}^p \lambda_{kk} (\textbf{u}_k^T \textbf{u}_k - 1).
\end{align*}
First differentiating w.r.t. $\textbf{U}$ and the Lagrangian multipliers in $\boldsymbol{\Lambda}$ and then noticing that the multipliers have two solutions that must be equal,
we get  equations
\[
\textbf{u}_l^T \textbf{T}_k =\textbf{u}_k^T \textbf{T}_l,\ \ k,l=1,...,p,
\]
where again
\begin{eqnarray*}
\textbf{T}_k &=& 3 \alpha E\left[(\textbf{u}_k^T \textbf{x}_{st})^3\right] E\left[(\textbf{u}_k^T \textbf{x}_{st})^2 \textbf{x}_{st} \right]\\
             &+& 4 (1 - \alpha) \left(E\left[(\textbf{u}_k^T \textbf{x}_{st})^4 \right] - 3\right)  E\left[(\textbf{u}_k^T \textbf{x}_{st})^3 \textbf{x}_{st} \right],\ \ k=1,...,p.
\end{eqnarray*}
If we then write
$$\textbf{T}= (\textbf{T}_1,..., \textbf{T}_p)^T$$
we get, as in \cite{miettinen2014fourth}, the following.

\begin{mylemma}\label{lem:estimeq}
The estimating equations for $\textbf{U}$ in Definition \ref{def:sym_func} are
\[\textbf{U} \textbf{T}^T = \textbf{T} \textbf{U}^T \ \ \mbox{and} \ \  \textbf{U}\textbf{U}^T = \textbf{I}_p \]
or, equivalently,  $\textbf{U}=\textbf{T} (\textbf{T}^T \textbf{T})^{-1/2}$.
\end{mylemma}

The estimating equations then suggest a fixed-point algorithm with a step
\[\textbf{U} \leftarrow  \textbf{T} (\textbf{T}^T \textbf{T})^{-1/2}. \]
and  further provide the following results regarding the asymptotic behavior of the estimate $\hat{\textbf{W}}$ in the case $\boldsymbol{\Omega} = \textbf{I}_p$. .

\begin{mytheorem}\label{theo:sym_asymp}
(i) Let $\textbf{z}_1,...,\textbf{z}_n$ be a random sample from a distribution with finite sixth moments and satisfying assumptions \ref{assu:ind} and \ref{assu:zeroskew}. Then there exists a sequence of solutions based on skewness (that is, $\alpha = 1$) such that $\hat{\textbf{W}}\rightarrow_P \textbf{I}_p$ and
\begin{align*}
\sqrt{n} (\hat{w}_{kk} - 1) &= -\dfrac{1}{2} \sqrt{n} (\hat{s}_{kk} - 1) + o_P(1), \\
\sqrt{n} \hat{w}_{kl} &= \frac{\sqrt{n} \hat{\psi}_{1kl}}{\gamma_k^2 + \gamma_l^2} + o_P(1), \quad l \neq k,
\end{align*}
where $\hat{\psi}_{1kl} = \gamma_k \hat{r}_{kl} - \gamma_l \hat{r}_{lk} - \gamma_k^2 \hat{s}_{kl}$.
\newline

(ii) Let $\textbf{z}_1,...,\textbf{z}_n$ be a random sample from a distribution with finite eighth moments and satisfying assumptions  \ref{assu:ind} and \ref{assu:zerokurt}. Then there exists a sequence of solutions based on kurtosis (that is, $\alpha = 0$) such that $\hat{\textbf{W}}\rightarrow_P \textbf{I}_p$ and
\begin{align*}
\sqrt{n} (\hat{w}_{kk} - 1) &= -\dfrac{1}{2} \sqrt{n} (\hat{s}_{kk} - 1) + o_P(1), \\
\sqrt{n} \hat{w}_{kl} &= \frac{\sqrt{n} \hat{\psi}_{2kl}}{\kappa_k^2 + \kappa_l^2} + o_P(1), \quad l \neq k,
\end{align*}
where $\hat{\psi}_{2kl} = \kappa_k \hat{q}_{kl} - \kappa_l \hat{q}_{lk} - (\kappa_k \beta_k - 3\kappa_l) \hat{s}_{kl}$.
\newline

(iii) Let $\textbf{z}_1,...,\textbf{z}_n$ be a random sample from a distribution with finite eighth moments and satisfying assumptions \ref{assu:ind} and \ref{assu:zeroboth}. Then there exists a sequence of solutions based on both skewness and kurtosis such that $\hat{\textbf{W}}\rightarrow_P \textbf{I}_p$ and
\begin{align*}
\sqrt{n} (\hat{w}_{kk} - 1) &= -\dfrac{1}{2} \sqrt{n} (\hat{s}_{kk} - 1) + o_P(1), \\
\sqrt{n} \hat{w}_{kl} &= \frac{3 \alpha \sqrt{n} \hat{\psi}_{1kl} + 4 (1 - \alpha)  \sqrt{n} \hat{\psi}_{2kl}}{3 \alpha (\gamma_k^2 + \gamma_l^2) + 4 (1 - \alpha) (\kappa_k^2 + \kappa_l^2)} + o_P(1), \quad l \neq k,
\end{align*}
where $\alpha \in [0, 1]$ is the  weight given to skewness, and $\hat{\psi}_{1kl}$ is as in (i) and $\hat{\psi}_{2kl}$ as in (ii).
\end{mytheorem}

\begin{mycorollary}\label{cor:sym_asymp}
(i) Under the assumptions of Theorem \ref{theo:sym_asymp}(i) the limiting distribution of $\sqrt{n} \, vec(\hat{\textbf{W}} - \textbf{I}_p)$ is multivariate normal with mean vector $\textbf{0}$ and the following asymptotic variances.
\begin{alignat*}{3}
&ASV&(\hat{w}_{kk}) &= \frac{\kappa_k + 2}{4}, &\\
&ASV&(\hat{w}_{kl}) &= \frac{\zeta_{11}}{(\gamma_k^2 + \gamma_l ^2)^2}, &\quad k \neq l,
\end{alignat*}
where $\zeta_{11} = \gamma_k^2 (\nu_k - \gamma_k^2) + \gamma_l^2 (\nu_l - \gamma_l^2) + \gamma_l^4$.
\newline

(ii) Under the assumptions of Theorem \ref{theo:sym_asymp}(ii) the limiting distribution of $\sqrt{n} \, vec(\hat{\textbf{W}} - \textbf{I}_p)$ is multivariate normal with mean vector $\textbf{0}$ and the following asymptotic variances.
\begin{alignat*}{3}
&ASV&(\hat{w}_{kk}) &= \frac{\kappa_k + 2}{4}, &\\
&ASV&(\hat{w}_{kl}) &= \frac{\zeta_{22}}{(\kappa_k^2 + \kappa_l ^2)^2}, &\quad k \neq l,
\end{alignat*}
where $\zeta_{22} = \kappa_k^2 (\omega_k - \beta_k^2) + \kappa_l^2 (\omega_l - \beta_l^2) + \kappa_l^4$.
\newline

(iii) Under the assumptions of Theorem \ref{theo:sym_asymp}(iii) the limiting distribution of $\sqrt{n} \, vec(\hat{\textbf{W}} - \textbf{I}_p)$ is multivariate normal with mean vector $\textbf{0}$ and the following asymptotic variances.
\begin{alignat*}{3}
&ASV&(\hat{w}_{kk}) &= \frac{\kappa_k + 2}{4}, &\\
&ASV&(\hat{w}_{kl}) &= \frac{9 \alpha^2 \zeta_{11} + 16 (1 - \alpha)^2 \zeta_{22} + 24 \alpha (1 - \alpha) \zeta_{12}}{(3 \alpha (\gamma_k^2 + \gamma_l^2) + 4 (1 - \alpha) (\kappa_k^2 + \kappa_l^2))^2}, &\quad k \neq l,
\end{alignat*}
where $\alpha \in [0, 1]$ is the weight given to skewness, and $\zeta_{11}$ is as in (i), $\zeta_{22}$ as in (ii) and $\zeta_{12} = \gamma_k \kappa_k (\eta_k - \gamma_k \beta_k) +  \gamma_l \kappa_l (\eta_l - \gamma_l \beta_l) + \gamma_l^2 \kappa_l^2$.
\end{mycorollary}

\section{Multivariate third and fourth cumulants}\label{sec:multi}

As the previous methods of estimating the unmixing matrix utilized only the marginal third and fourth cumulants of the components, a natural question is whether the use of multivariate moments has any benefits.  We therefore consider the following sets of matrices, capturing all joint third and fourth cumulants  of the random $p$-vector  $\textbf{x}$  with $E(\textbf{x})=\textbf{0}$.
\begin{eqnarray*}
\textbf{C}^{3i}(\textbf{x}) & =&  E\left[ x_i \cdot \textbf{x} \textbf{x}^T \right], \quad i=1,...,p, \ \mbox{and}\\
\textbf{C}^{4ij}(\textbf{x}) & =&  E\left[ x_i x_j \cdot\textbf{x}  \textbf{x}^T \right]- E(x_i x_j) E(\textbf{x} \textbf{x}^T)  \\
                             & -& E(x_i \cdot \textbf{x}) E(x_j \cdot \textbf{x}^T) -  E(x_j \cdot \textbf{x})E(x_i \cdot \textbf{x}^T), \quad i,j=1,...,p.
%\textbf{B}^{ij}(\textbf{x}) & =  E\left[ x_i x_j \cdot\textbf{x}  \textbf{x}^T \right], \quad &i,j=1,...,p.
\end{eqnarray*}
Evaluating matrices $\textbf{C}^{3i}(\textbf{z})$ and $\textbf{C}^{4ij}(\textbf{z})$ gives
\[\textbf{C}^{3i}(\textbf{z}) = \gamma_i \textbf{E}^{ii} \quad \text{and} \quad \textbf{C}^{4ij}(\textbf{z}) = \delta_{ij} \kappa_i \textbf{E}^{ii}, \]
showing that both $\textbf{C}^{3i}(\textbf{z})$ and $\textbf{C}^{4ij}(\textbf{z})$ are diagonal for all $i,j=1,...,p$. Based upon them we can construct two matrices combining specific subsets of third and fourth joint cumulants, which we will call \textit{compound cumulant matrices}:
\[
\textbf{C}^3(\textbf{z}) = \sum_{i=1}^p \textbf{C}^{3i}(\textbf{z}) = \sum_{i=1}^p \gamma_i \textbf{E}^{ii} \quad \text{and} \quad \textbf{C}^4(\textbf{z}) = \sum_{i=1}^p \textbf{C}^{4ii}(\textbf{z}) = \sum_{i=1}^p \kappa_i \textbf{E}^{ii}. \]
%the second of which has the property $\textbf{C}^4(\textbf{x}_{st}) = E(\textbf{x}_{st} \textbf{x}_{st}^T \textbf{x}_{st} \textbf{x}_{st}^T) - (p+2) \textbf{I}_p$.
The next theorem then gives us two viable ways of recovering the independent components using the previously defined cumulant matrices.
\begin{mytheorem}\label{theo:diag}
Let $\textbf{z} \in \mathbb{R}^p$ have independent components with $E(\textbf{z}) = \textbf{0}$ and $Cov(\textbf{z}) = \textbf{I}_p$.
Then for all orthogonal $\textbf{U}$, the eigenvectors of the symmetric matrices
$\textbf{C}^{\,3i}( \textbf{U}\textbf{z})$, $\textbf{C}^{\,4ij}(\textbf{U}\textbf{z})$, $\textbf{C}^{\,3}(\textbf{U}\textbf{z})$ and $\textbf{C}^{\,4}(\textbf{U}\textbf{z})$, $i,j=1,...,p$ are the columns of $\textbf{U}$.
\end{mytheorem}
Theorem \ref{theo:diag} says that the rotation giving the independent components from the standardized observations is such that it diagonalizes all the matrices $\textbf{C}^3(\textbf{x}_{st}), \textbf{C}^4(\textbf{x}_{st}), \textbf{C}^{3i}(\textbf{x}_{st})$ and $\textbf{C}^{4ij}(\textbf{x}_{st})$, $i,j=1,...,p$. To recover the independent components in practice we thus want to find a rotation $\textbf{U} \in \mathcal{U}$ that simultaneously makes all the matrices $\textbf{U}\textbf{C}_s\textbf{U}^T$, where $\textbf{C}_s,\, s=1,...,S$, is some subset of the previous matrices, as diagonal as possible.

One way of accomplishing this is based on the observation that for any family of matrices $\textbf{C}_s, s=1,...,S$, and any $\textbf{U} \in \mathcal{U}$ we have
\[\sum_{s=1}^S\|\text{diag}(\textbf{U}\textbf{C}_s\textbf{U}^T)\|^2 + \sum_{s=1}^S\|\text{off}(\textbf{U}\textbf{C}_s\textbf{U}^T)\|^2 = \sum_{s=1}^S \| \textbf{C}_s \|^2, \]
implying that the joint approximate diagonalization can be preformed by finding $\textbf{U} \in \mathcal{U}$ that maximizes the sum $\sum_{s=1}^S\|\text{diag}(\textbf{U}\textbf{C}_s\textbf{U}^T)\|^2$. The process can then be thought as a sort of ``joint eigendecomposition''. The concept is not new and has been used before e.g. in \citet{cardoso1993blind} and in \citet{moreau2001generalization}.

%or equivalently in matrix form
%\[\textbf{C}^{4ij} = \textbf{B}^{ij} - \textbf{E}^{ij} - (\textbf{E}^{ij})^T - \delta_{ij} \textbf{I}_p, \quad i,j=1,...,p. \]
%It is then clear that the matrices $\textbf{C}^{3i} \in \mathbb{R}^{p \times p}, \, i=1,...,p$ together contain all possible multivariate third cumulants and the matrices $\textbf{C}^{4ij} \in \mathbb{R}^{p \times p}, \, i,j=1,...,p$ likewise all the fourth cumulants. The JADE-type methods in Section \ref{subsec:jade} are based on the use of these families of matrices.

%Furthermore, using the previous matrices we define the matrices

%the latter of which has the equivalent form $\textbf{C}^4 = E\left[ \textbf{x}_{st} \textbf{x}_{st}^T \textbf{x}_{st} \textbf{x}_{st}^T \right] - (p+2)\textbf{I}_p$. These matrices containing a specific subset of the multivariate third and fourth moments are used in the FOBI-type methods in Section \ref{subsec:fobi}.

\subsection{Using compound cumulant matrices}\label{subsec:fobi}

Having already justified the working of the following methods, we first present the use of compound cumulant matrices $\textbf{C}^3$ and $\textbf{C}^4$ in recovering the independent components.

In order to obtain an affine equivariant procedure this time,  we must use a somewhat unorthodox standardization. Namely, we pretransform the data by an arbitrary independent component functional.  This is necessitated by the ``bad behavior'' of the compound matrix of third cumulants $\textbf{C}^3$. Writing the IC functional in the form $(\boldsymbol{\Sigma}^*)^{-1/2}=\textbf{U}^* \boldsymbol{\Sigma}^{-1/2}$ for some $\textbf{U}^* \in \mathcal{U}$, makes the standardization then correspond to the transformation $\textbf{x} \mapsto \textbf{x}_{st}^* = (\boldsymbol{\Sigma}^*)^{-1/2} ( \textbf{x} - \boldsymbol{\mu})$. Note that $(\boldsymbol{\Sigma}^*)^{-1/2}\boldsymbol{\Sigma}(\boldsymbol{\Sigma}^*)^{-1/2}=\textbf{I}_p$ so that $(\boldsymbol{\Sigma}^*)^{-1/2}$ is just a certain asymmetric version of $\boldsymbol{\Sigma}^{-1/2}$. Surprisingly, the limiting behavior of the estimates then does not depend on the root-$n$ consistent choice of $(\boldsymbol{\Sigma}^*)^{-1/2}$.
 Note that this idea to achieve affine equivariance for another IC method was also used in \cite{MiettinenNordhausenOjaTaskinen:2013}.

%Namely, we transform the data into an invariant coordinate system given by $Cov(\textbf{x})$ and an arbitrary scatter functional $\textbf{S}(F_\textbf{x})$, see \citet{TylerCritchleyDumbgenOja:2009}. This procedure is necessitated by the ``bad behavior'' of the compound matrix of third cumulants $\textbf{C}^3$. The standardization then corresponds to the transformation $\textbf{x} \mapsto \textbf{x}_{st}^* = \textbf{F} \boldsymbol{\Sigma}^{-1/2} ( \textbf{x} - \boldsymbol{\mu})$, where $\textbf{F} \in \mathcal{U}$ is the orthogonal matrix obtained from the invariant coordinate selection. Note that this idea to achieve affine equivariance for another IC method was also used in \cite{MiettinenNordhausenOjaTaskinen:2013}.

\begin{mydefinition} \label{def:fobi_func}
The compound cumulant functional based on both third and fourth cumulants is a functional $\textbf{W} (F_\textbf{x}) = \textbf{U} (\boldsymbol{\Sigma}^*)^{-1/2}$, where $(\boldsymbol{\Sigma}^*)^{-1/2}$ is the standardizing IC functional and the orthogonal matrix $\textbf{U}$ is found as
\[ \textbf{U} = \underset{\textbf{U} \in \mathcal{U}}{\emph{argmax}} \left( \alpha \| diag( \textbf{U} \textbf{C}^{\,3}(\textbf{x}_{st}^*) \textbf{U}^T ) \|^2 + (1 - \alpha) \| diag( \textbf{U} \textbf{C}^{\,4}(\textbf{x}_{st}^*) \textbf{U}^T ) \|^2 \right), \]
where $\alpha \in [0, 1]$ is the proportion of weight given to skewness.
\end{mydefinition}

Letting then $\alpha = 1$ or $\alpha = 0$ and using the properties of the matrices $\textbf{C}^3$ and $\textbf{C}^4$ yields the following corollary.

\begin{mycorollary}\label{cor:fobi_xor}
(i) The compound cumulant functional  based on third cumulants (that is, $\alpha = 1$) is  a functional  $\textbf{W} (F_\textbf{x}) = \textbf{U} (\boldsymbol{\Sigma}^*)^{-1/2}$, where   $\textbf{U}$ has the eigenvectors of $\textbf{C}^{\,3}(\textbf{x}_{st}^*)$ as its rows.
\newline
(ii) The compound cumulant functional  based on fourth cumulants (that is, $\alpha = 0$) is a functional  $\textbf{W} (F_\textbf{x}) = \textbf{U} (\boldsymbol{\Sigma}^*)^{-1/2}$, where  $\textbf{U}$ has the eigenvectors of $\textbf{C}^{\,4}(\textbf{x}_{st}^*)$ as its rows.
\end{mycorollary}

%Observe that the solution in Corollary \ref{cor:fobi_xor}(ii) is obtained already in the standardization phase, $\textbf{x} \mapsto \textbf{x}_{st}^*$.

%Note also that although in the classic FOBI the solution is found by diagonalizing the matrix $E\left[ \textbf{x}_{st} \textbf{x}_{st}^T \textbf{x}_{st} \textbf{x}_{st}^T \right]$, using the cumulant version of the matrix, $\textbf{C}^4(\textbf{x}_{st}) = E\left[ \textbf{x}_{st} \textbf{x}_{st}^T \textbf{x}_{st} \textbf{x}_{st}^T \right] - (p+2)\textbf{I}_p$, would also lead to the same solution.

As already stated, the different standardization mechanism guarantees the affine equivariance of the procedure.

%, however, we have to assume that the eigenbasis consisting of the rows of $\textbf{F}$ is unique up to sign and order.

\begin{mylemma}\label{lem:fobi_ae}
The compound cumulant functional $\textbf{W}(F_\textbf{x})$ in Definition \ref{def:fobi_func} is an independent component functional for every $\alpha \in [0, 1]$.
\end{mylemma}

\begin{myremark}
We implicitly assume here that the IC functional used in the standardization exists and is well-defined. The extra assumptions needed for its existence and root-$n$ consistency can be seen as the price we have to pay for making the compound cumulant method affine equivariant.
\end{myremark}

\begin{myremark}
If $\alpha=0$ and only $ \textbf{C}^4$ is used, the prestandardization is not needed and the classical FOBI estimate is obtained as the solution.
The most natural choice for $(\boldsymbol{\Sigma}^*)^{-1/2}$ is then the FOBI functional.
\end{myremark}

The Lagrangian of the objective function in the maximization problem of Definition \ref{def:fobi_func} then has the form
\begin{align*}
 L(\textbf{U}, \boldsymbol{\Lambda}) &= \alpha \sum_{k=1}^p (\textbf{u}_k^T \textbf{C}^3\textbf{u}_k)^2 + (1 - \alpha) \sum_{k=1}^p (\textbf{u}_k^T \textbf{C}^4 \textbf{u}_k)^2 \\
 &- \sum_{k=1}^{p-1} \sum_{l=k+1}^p \lambda_{kl} \textbf{u}_k^T \textbf{u}_l - \sum_{k=1}^p \lambda_{kk} (\textbf{u}_k^T \textbf{u}_k - 1).
\end{align*}
(The matrices $\textbf{C}^3$ and  $\textbf{C}^4$ are evaluated at $\textbf{x}_{st}^*$.)
The optimization can be done as in Section \ref{symm FastIca}, and the estimation equations are
\[
\textbf{U}\textbf{T}^T=\textbf{T}\textbf{U}^T\ \ \mbox{and}\ \ \textbf{U}\textbf{U}^T=\textbf{I}_p
\]
where $\textbf{T}=(\textbf{T}_1,...,\textbf{T}_p)^T$ with
\[\textbf{T}_k = \alpha (\textbf{u}_k^T \textbf{C}^3 \textbf{u}_k) \textbf{C}^3 \textbf{u}_k + (1 - \alpha)(\textbf{u}_k^T \textbf{C}^4 \textbf{u}_k) \textbf{C}^4 \textbf{u}_k,
\ \ k=1,...,p.\]
The estimating equations again suggest a fixed-point algorithm and can be used to find the asymptotic behaviors of the estimates. We then have the following.

\begin{mytheorem}\label{theo:fobi_asymp}
(i) Let $\textbf{z}_1,...,\textbf{z}_n$ be a random sample from a distribution with finite sixth moments and satisfying assumptions \ref{assu:ind} and \ref{assu:distskew}. Assume further that $\sqrt{n}\left((\boldsymbol{\Sigma}^*)^{-1/2} - \textbf{I}_p\right)=O_p(1)$. Then there exists a sequence of solutions based on skewness (that is, $\alpha = 1$) such that $\hat{\textbf{W}}\rightarrow_P \textbf{I}_p$ and
\begin{align*}
\sqrt{n} (\hat{w}_{kk} - 1) &= -\dfrac{1}{2} \sqrt{n} (\hat{s}_{kk} - 1) + o_P(1), \\
\sqrt{n} \hat{w}_{kl} &= \dfrac{\sqrt{n} \hat{\psi}_{1kl}}{\gamma_k - \gamma_l}
 + o_P(1), \quad k \neq l,
\end{align*}
where $\hat{\psi}_{1kl} = \hat{r}_{kl} + \hat{r}_{lk} + \sum_{m \neq k,l} \hat{r}_{mkl} - \gamma_k \hat{s}_{kl}$.
\newline

(ii) Let $\textbf{z}_1,...,\textbf{z}_n$ be a random sample from a distribution with finite eighth moments and satisfying assumptions \ref{assu:ind} and \ref{assu:distkurt}. Assume further that $\sqrt{n}\left((\boldsymbol{\Sigma}^*)^{-1/2} - \textbf{I}_p\right)=O_p(1)$. Then there exists a sequence of solutions based on kurtosis (that is, $\alpha = 0$) such that $\hat{\textbf{W}}\rightarrow_P \textbf{I}_p$ and
\begin{align*}
\sqrt{n} (\hat{w}_{kk} - 1) &= -\dfrac{1}{2} \sqrt{n} (\hat{s}_{kk} - 1) + o_P(1), \\
\sqrt{n} \hat{w}_{kl} &= \dfrac{\sqrt{n} \hat{\psi}_{2kl}}{\kappa_k - \kappa_l}
 + o_P(1), \quad k \neq l,
\end{align*}
where $\hat{\psi}_{2kl} = \hat{q}_{kl} + \hat{q}_{lk} + \sum_{m \neq k,l} \hat{q}_{mkl} - (\kappa_k + p + 4) \hat{s}_{kl}$.
\newline

(iii) Let $\textbf{z}_1,...,\textbf{z}_n$ be a random sample from a distribution with finite eighth moments and satisfying assumptions \ref{assu:ind} and \ref{assu:distboth}. Assume further that $\sqrt{n}\left((\boldsymbol{\Sigma}^*)^{-1/2} - \textbf{I}_p\right)=O_p(1)$. Then there exists a sequence of solutions based on both skewness and kurtosis such that $\hat{\textbf{W}}\rightarrow_P \textbf{I}_p$ and
\begin{align*}
\sqrt{n} (\hat{w}_{kk} - 1) &= -\dfrac{1}{2} \sqrt{n} (\hat{s}_{kk} - 1) + o_P(1), \\
\sqrt{n} \hat{w}_{kl} &= \frac{\alpha(\gamma_k - \gamma_l) \sqrt{n} \hat{\psi}_{1kl} + (1 - \alpha) (\kappa_k - \kappa_l) \sqrt{n} \hat{\psi}_{2kl}}{\alpha (\gamma_k - \gamma_l)^2 + (1 - \alpha)(\kappa_k - \kappa_l)^2} + o_P(1), \quad k \neq l,
\end{align*}
where $\alpha \in [0, 1]$ is the proportion of weight given to skewness, and $\hat{\psi}_{1kl}$ is as in (i) and $\hat{\psi}_{2kl}$ as in (ii).
\end{mytheorem}

\begin{mycorollary} \label{cor:fobi_asymp}
(i) Under the assumptions of Theorem \ref{theo:fobi_asymp}(i) the limiting distribution of $\sqrt{n} \, vec(\hat{\textbf{W}} - \textbf{I}_p)$ is multivariate normal with mean vector $\textbf{0}$ and the following asymptotic variances.
\begin{alignat*}{3}
&ASV&(\hat{w}_{kk}) &= \frac{\kappa_k + 2}{4}, &\\
&ASV&(\hat{w}_{kl}) &= \frac{\zeta_{11}}{(\gamma_k - \gamma_l)^2}, &\quad k \neq l,
\end{alignat*}
where $\zeta_{11} = (\nu_k - \gamma_k^2) + (\nu_l - \gamma_l^2) + \gamma_l^2 + (p-2)$.

(ii) Under the assumptions of Theorem \ref{theo:fobi_asymp}(ii) the limiting distribution of $\sqrt{n} \, vec(\hat{\textbf{W}} - \textbf{I}_p)$ is multivariate normal with mean vector $\textbf{0}$ and the following asymptotic variances.
\begin{alignat*}{3}
&ASV&(\hat{w}_{kk}) &= \frac{\kappa_k + 2}{4}, &\\
&ASV&(\hat{w}_{kl}) &= \frac{\zeta_{22}}{(\kappa_k - \kappa_l)^2}, &\quad k \neq l,
\end{alignat*}
where $\zeta_{22} = (\omega_k - \beta_k^2) + (\omega_l - \beta_l^2) + \kappa_l^2 + \sum_{m \neq k,l} (\beta_m - 1)$.

(iii) Under the assumptions of Theorem \ref{theo:fobi_asymp}(iii) the limiting distribution of $\sqrt{n} \, vec(\hat{\textbf{W}} - \textbf{I}_p)$ is multivariate normal with mean vector $\textbf{0}$ and the following asymptotic variances.
\begin{alignat*}{3}
&ASV&(\hat{w}_{kk}) &= \frac{\kappa_k + 2}{4}, &\\
&ASV&(\hat{w}_{kl}) &= \frac{\alpha^2\delta_{1kl}^2 \zeta_{11} + (1 - \alpha)^2\delta_{2kl}^2 \zeta_{22} + 2 \alpha (1 - \alpha) \delta_{1kl} \delta_{2kl} \zeta_{12}}{(\alpha \delta_{1kl}^2 + (1 - \alpha) \delta_{2kl}^2)^2}, &\quad k \neq l,
\end{alignat*}
where $\delta_{1kl} = (\gamma_k - \gamma_l)$, $\delta_{2kl} = (\kappa_k - \kappa_l)$, $\alpha \in [0, 1]$ is the proportion of weight given to skewness, and $\zeta_{11}$ is as in (i), $\zeta_{22}$ as in (ii) and $\zeta_{12} = (\eta_k - \gamma_k \beta_k) + (\eta_l - \gamma_l \beta_l) + \gamma_l \kappa_l + \sum_{m \neq k,l} \gamma_m$.
\end{mycorollary}

By noting in Corollary \ref{cor:fobi_asymp}(ii) that $(\beta_m - 1) > 0, \, m=1,...,p$ we further get a lower bound for the corresponding asymptotic variance.

\begin{mycorollary}
The asymptotic variance of the classical FOBI in \ref{cor:fobi_asymp}(ii) has a lower bound of
\[\frac{(\omega_k - \beta_k^2) + (\omega_l - \beta_l^2) + \kappa_l^2}{(\kappa_k - \kappa_l)^2}. \]
\end{mycorollary}

%\begin{myremark}
%Note that for ``adding'' the affine equivariance to the method via the standardization by $\textbf{F}$, we have to pay the price of assuming distinct kurtosis values for the independent components. This especially means that even when using only third cumulants ($\alpha = 1$) we still have to assume the existence of eighth moments, and furthermore that the excess kurtoses are distinct.
%\end{myremark}

%QUESTION: COULD WE GET HERE A REMARK TELLING THE "COSTS" OF USING FOBI INSTEAD OF WHITENING???

\subsection{Using all cumulant matrices}\label{subsec:jade}

Besides the issue of achieving affine equivariance, another clear drawback of using the compound cumulant matrices $\textbf{C}^3$ and $\textbf{C}^4$ in the estimation of the unmixing matrix is that these matrices combine only certain  subsets of all $p^3$ or $p^4$ possible joint third or fourth cumulants. As such, the compound cumulant method may not use all the information available in joint cumulants.

A standard solution used in the literature (for fourth cumulants) is to, instead of using only the matrix $\textbf{C}^4$, use all the cumulant matrices $\textbf{C}^{\,4ij}$ simultaneously. This approach, called JADE, was introduced in \cite{cardoso1993blind}. For further details and variants, see also \cite{BonhommeRobin2009,MiettinenNordhausenOjaTaskinen:2013}.
We give the following.

\begin{mydefinition}\label{def:jade_func}
The functional  based on all third and fourth cumulants is a functional  $\textbf{W} (F_\textbf{x}) = \textbf{U}\boldsymbol{\Sigma}^{-1/2}$, where $\boldsymbol{\Sigma} = Cov(\textbf{x})$ and
\[ \textbf{U} = \underset{\textbf{U} \in \mathcal{U}}{\emph{argmax}} \left( \alpha \sum_{i=1}^p \| diag(\textbf{U} \textbf{C}^{\,3i} \textbf{U}^T)  \|^2 + (1 - \alpha) \sum_{i=1}^p \sum_{j=1}^p \| diag(\textbf{U} \textbf{C}^{\,4ij} \textbf{U}^T)  \|^2 \right), \]
where $\textbf{C}^{\,3i}$ and $\textbf{C}^{\,4ij}$ are evaluated at $\textbf{x}_{st}$ and $\alpha \in [0, 1]$ is the proportion of weight given to skewness.
\end{mydefinition}
For $\alpha=0$, the classical JADE estimate is obtained. \citet{moreau2001generalization} has a similar definition for his \textit{eJADE} estimate, the matrices $\textbf{U} \textbf{C}^{\,3i}(\textbf{x}_{st}) \textbf{U}^T$ replaced by $\textbf{C}^{3i}(\textbf{U}\textbf{x}_{st})$.
%Extensions of JADE for cumulants of arbitrary order $r \geq 3$ and positive linear combinations of them were considered already in \citet{moreau2001generalization} via a method denoted \textit{eJADE}. Our proposal of using third and fourth joint cumulants in the next definition, while slightly differing from their formulation, can be shown to provide the same solution as the corresponding eJADE.
%The corresponding version of eJADE proposed in \citet{moreau2001generalization} is now obtained from Definition \ref{def:jade_func} by replacing the matrices $\textbf{U} \textbf{C}^{\,3i}(\textbf{x}_{st}) \textbf{U}^T$ with $\textbf{C}^{3i}(\textbf{U}\textbf{x}_{st})$.
\cite{miettinen2014fourth} proved that the classical JADE functional ($\alpha=0$) is affine equivariant. This is true for the combined functional $\textbf{W} (F_\textbf{x})$ as well and we have the following.

\begin{mylemma}\label{lem:jade_ae}
The functional $\textbf{W} (F_\textbf{x})$ in Definition \ref{def:jade_func} is an independent component functional  for all $\alpha \in [0, 1]$.
\end{mylemma}

The Lagrangian of the maximization problem in Definition \ref{def:jade_func} has the form
\begin{align*}
L(\textbf{U}, \boldsymbol{\Lambda}) &= \alpha \sum_{i=1}^p \sum_{k=1}^p (\textbf{u}_k^T \textbf{C}^{3i} \textbf{u}_k )^2 + (1 - \alpha) \sum_{i=1}^p \sum_{j=1}^p \sum_{k=1}^p (\textbf{u}_k^T \textbf{C}^{4ij} \textbf{u}_k)^2 \\
&- \sum_{k=1}^{p-1} \sum_{l=k+1}^p \lambda_{kl} \textbf{u}_k^T \textbf{u}_l - \sum_{k=1}^p \lambda_{kk} (\textbf{u}_k^T \textbf{u}_k - 1).
\end{align*}
Optimization provides the estimating equations
\[
\textbf{U}\textbf{T}^T=\textbf{T}\textbf{U}^T\ \ \mbox{and}\ \ \textbf{U}\textbf{U}^T=\textbf{I}_p
\]
where $\textbf{T}=(\textbf{T}_1,...,\textbf{T}_p)^T$ now with
\[\textbf{T}_k = \alpha \sum_{i=1}^p (\textbf{u}_k^T \textbf{C}^{3i} \textbf{u}_k) \textbf{C}^{3i} \textbf{u}_k + (1 - \alpha) \sum_{i=1}^p \sum_{j=1}^p (\textbf{u}_k^T \textbf{C}^{4ij} \textbf{u}_k) \textbf{C}^{4ij} \textbf{u}_k.\]
As in previous sections we then obtain the following.
\begin{mytheorem}\label{theo:jade_asymp}
(i) Let $\textbf{z}_1,...,\textbf{z}_n$ be a random sample from a distribution with finite sixth moments and satisfying assumptions \ref{assu:ind} and \ref{assu:zeroskew}. Then there exists a sequence of solutions based on skewness (that is, $\alpha = 1$) such that $\hat{\textbf{W}}\rightarrow_P \textbf{I}_p$ and
\begin{align*}
\sqrt{n} (\hat{w}_{kk} - 1) &= -\dfrac{1}{2} \sqrt{n} (\hat{s}_{kk} - 1) + o_P(1), \\
\sqrt{n} \hat{w}_{kl} &= \frac{\sqrt{n} \hat{\psi}_{1kl}}{\gamma_k^2 + \gamma_l^2} + o_P(1), \quad l \neq k,
\end{align*}
where $\hat{\psi}_{1kl} = \gamma_k \hat{r}_{kl} - \gamma_l \hat{r}_{lk} - \gamma_k^2 \hat{s}_{kl}$.
\newline

(ii) Let $\textbf{z}_1,...,\textbf{z}_n$ be a random sample from a distribution with finite eighth moments and satisfying assumptions \ref{assu:ind} and \ref{assu:zerokurt}. Then there exists a sequence of solutions based on kurtosis (that is, $\alpha = 0$) such that $\hat{\textbf{W}}\rightarrow_P \textbf{I}_p$ and
\begin{align*}
\sqrt{n} (\hat{w}_{kk} - 1) &= -\dfrac{1}{2} \sqrt{n} (\hat{s}_{kk} - 1) + o_P(1), \\
\sqrt{n} \hat{w}_{kl} &= \frac{\sqrt{n} \hat{\psi}_{2kl}}{\kappa_k^2 + \kappa_l^2} + o_P(1), \quad l \neq k,
\end{align*}
where $\hat{\psi}_{2kl} = \kappa_k \hat{q}_{kl} - \kappa_l \hat{q}_{lk} - (\kappa_k \beta_k - 3\kappa_l) \hat{s}_{kl}$.
\newline

(iii) Let $\textbf{z}_1,...,\textbf{z}_n$ be a random sample from a distribution with finite eighth moments and satisfying assumptions \ref{assu:ind} and \ref{assu:zeroboth}. Then there exists a sequence of solutions based on both skewness and kurtosis such that $\hat{\textbf{W}}\rightarrow_P \textbf{I}_p$ and
\begin{align*}
\sqrt{n} (\hat{w}_{kk} - 1) &= -\dfrac{1}{2} \sqrt{n} (\hat{s}_{kk} - 1) + o_P(1), \\
\sqrt{n} \hat{w}_{kl} &= \frac{\alpha \sqrt{n} \hat{\psi}_{1kl} + (1 - \alpha) \sqrt{n} \hat{\psi}_{2kl}}{\alpha (\gamma_k^2 + \gamma_l^2) + (1 - \alpha) (\kappa_k^2 + \kappa_l^2)} + o_P(1), \quad l \neq k,
\end{align*}
where $\alpha \in [0, 1]$ is the proportion of weight given to skewness, and $\hat{\psi}_{1kl}$ is as in (i) and $\hat{\psi}_{2kl}$ as in (ii).
\end{mytheorem}

\begin{mycorollary} \label{cor:jade_asymp}
(i) Under the assumptions of Theorem \ref{theo:jade_asymp}(i) the limiting distribution of $\sqrt{n} \, vec(\hat{\textbf{W}} - \textbf{I}_p)$ is multivariate normal with mean vector $\textbf{0}$ and the following asymptotic variances.
\begin{alignat*}{3}
&ASV&(\hat{w}_{kk}) &= \frac{\kappa_k + 2}{4}, &\\
&ASV&(\hat{w}_{kl}) &= \frac{\zeta_{11}}{(\gamma_k^2 + \gamma_l ^2)^2}, &\quad k \neq l,
\end{alignat*}
where $\zeta_{11} = \gamma_k^2 (\nu_k - \gamma_k^2) + \gamma_l^2 (\nu_l - \gamma_l^2) + \gamma_l^4$.
\newline

(ii) Under the assumptions of Theorem \ref{theo:jade_asymp}(ii) the limiting distribution of $\sqrt{n} \, vec(\hat{\textbf{W}} - \textbf{I}_p)$ is multivariate normal with mean vector $\textbf{0}$ and the following asymptotic variances.
\begin{alignat*}{3}
&ASV&(\hat{w}_{kk}) &= \frac{\kappa_k + 2}{4}, &\\
&ASV&(\hat{w}_{kl}) &= \frac{\zeta_{22}}{(\kappa_k^2 + \kappa_l ^2)^2}, &\quad k \neq l,
\end{alignat*}
where $\zeta_{22} = \kappa_k^2 (\omega_k - \beta_k^2) + \kappa_l^2 (\omega_l - \beta_l^2) + \kappa_l^4$.
\newline

(iii) Under the assumptions of Theorem \ref{theo:jade_asymp}(iii) the limiting distribution of $\sqrt{n} \, vec(\hat{\textbf{W}} - \textbf{I}_p)$ is multivariate normal with mean vector $\textbf{0}$ and the following asymptotic variances.
\begin{alignat*}{3}
&ASV&(\hat{w}_{kk}) &= \frac{\kappa_k + 2}{4}, &\\
&ASV&(\hat{w}_{kl}) &= \frac{\alpha^2 \zeta_{11} + (1 - \alpha)^2 \zeta_{22} + 2 \alpha (1 - \alpha) \zeta_{12}}{(\alpha (\gamma_k^2 + \gamma_l^2) + (1 - \alpha) (\kappa_k^2 + \kappa_l^2))^2}, &\quad k \neq l,
\end{alignat*}
where $\alpha \in [0, 1]$ is the proportion of weight given to skewness, and $\zeta_{11}$ is as in (i), $\zeta_{22}$ as in (ii) and $\zeta_{12} = \gamma_k \kappa_k (\eta_k - \gamma_k \beta_k) + \gamma_l \kappa_l (\eta_l - \gamma_l \beta_l) + \gamma_l^2 \kappa_l^2$.
\end{mycorollary}

Comparison of the asymptotic variances in Corollary \ref{cor:sym_asymp} and Corollary \ref{cor:jade_asymp} immediately yields the following result.

\begin{mycorollary}
(i) If $\alpha=0$ or $\alpha=1$, the asymptotic variances of the estimates based on symmetric projection pursuit are the same as those based on all cumulant matrices.
\newline
(ii)
The asymptotic variances of the estimates based on symmetric projection pursuit are the same as those based on all cumulant matrices
if their respective weights $\alpha_{S}$ and $\alpha_{J}$ satisfy $\alpha_{S} = 4 \alpha_J/(3 + \alpha_J)$.
\end{mycorollary}

\section{Comparison of the estimates} \label{sec:simu}

\subsection{Projection pursuit for cluster identification}
Given that all the methods allow tuning in the form of the weighting parameter $\alpha$, a natural question is whether there exists some optimal choice of weighting for any particular pair of source distributions. We approach this question first in the context of  cluster identification. This approach is not new, for both skewness and kurtosis have been used before for similar purposes. In \cite{jones1987projection} the authors use approximative techniques to find a linear combination of squared skewness and kurtosis to use as \textit{entropy index} in projection pursuit. In \cite{pena2001cluster} the authors project the data into directions that have extremal kurtosis in hopes of discovering clusters.

For the model, assume that the vector of independent components $\textbf{z}$ is a mixture of two multivariate normal distributions $\textbf{z}^*$, namely
\[\textbf{z}^* \sim \pi \cdot \mathcal{N}_p(\textbf{0}, \textbf{I}_p) + (1-\pi) \cdot \mathcal{N}_p(\mu \textbf{e}_1, \textbf{I}_p),\]
standardized to have zero mean and identity covariance matrix and where $\pi \in (0, 1)$ and $\mu \in \mathbb{R}\backslash \{0\}$. Under the model, the only independent component having non-zero skewness or kurtosis is the first one, meaning that only the first row of the unmixing matrix is identifiable (up to sign). However, this is enough as the first component carries all the information needed for the group separation, the remaining components can be considered just as noise.

We use the projection pursuit approach to estimate the direction $\textbf{w}$ of Fisher's linear discriminant  subspace. Note next that $\gamma_l=\kappa_l=0$ for $l>1$ and  the asymptotic variances of the estimated elements $\hat{w}_{1l},\, l>1,$ are the same for the deflation-based and symmetric projection pursuit. Thus to choose the optimal weighting for group separation we want to minimize the variance
\[f(\alpha; \pi, \mu) := \frac{9 \alpha^2 \zeta_{11k} + 16 (1 - \alpha)^2 \zeta_{22k} + 24 \alpha (1 - \alpha) \zeta_{12k}}{(3 \alpha \gamma_k^2 + 4 (1 - \alpha) \kappa_k^2)^2}, \quad \alpha  \in [0, 1],\]
for $k = 1$, where $\zeta_{11k}, \zeta_{22k}$ and $\zeta_{12k}$ are as in Corollary \ref{cor:sep_asymp}.

\begin{figure}[t]
    \centering
    \includegraphics[width=1\textwidth]{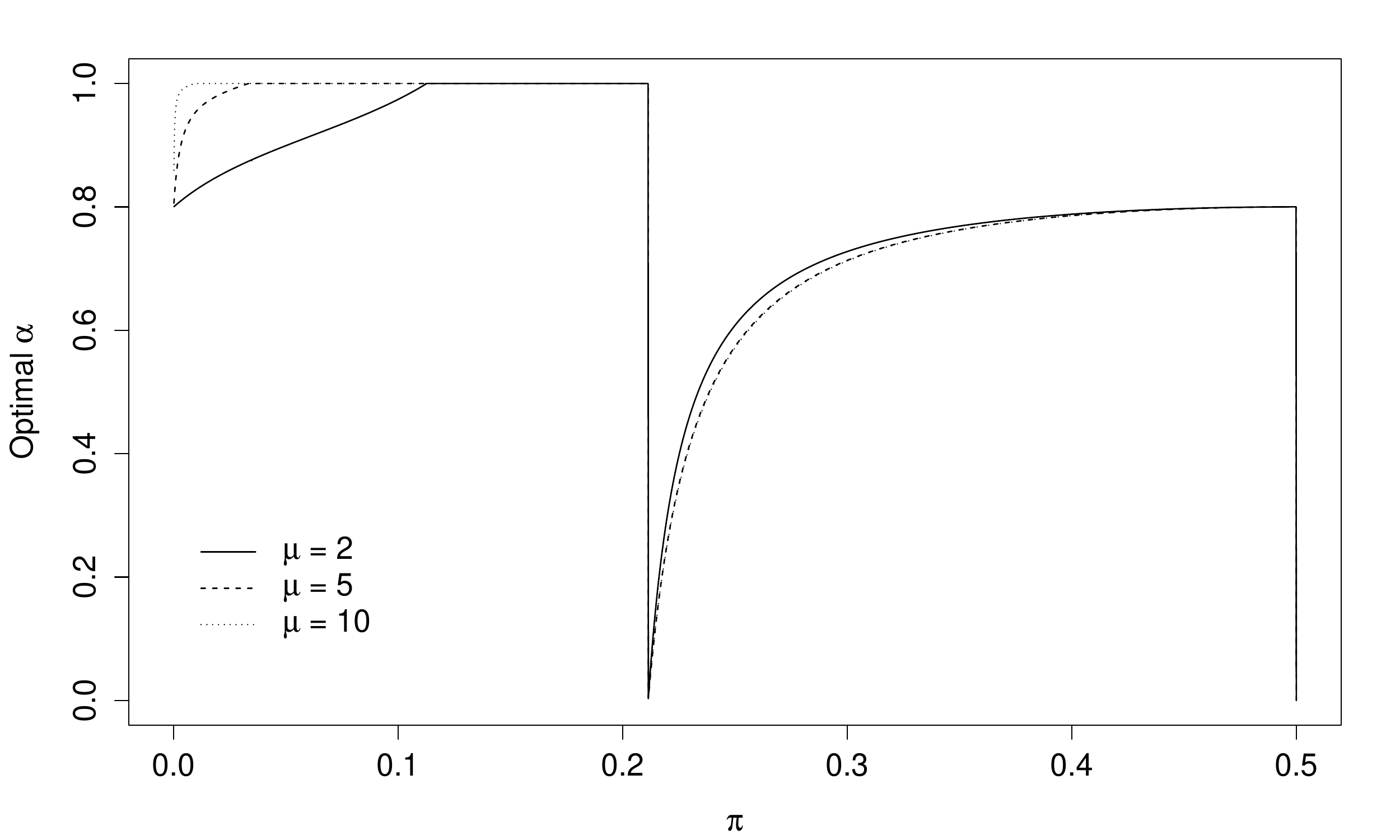}
    \caption{The optimal choices of weight $\alpha$ for different values of $\pi$ and $\mu$.}
    \label{fig:sim_discr}
\end{figure}

Using the function \textit{optimize} in R \citep{Rcore} we searched the global minimum of $f$ for three choices of $\mu = 2, 5, 10$ and the results are shown in Figure \ref{fig:sim_discr} (we only need to consider the interval $\pi \in (0, 0.5]$ due to symmetry). First, the plots show that the choice of $\mu$ has hardly any effect on the optimal value of $\alpha$. Secondly, we see two discontinuity points, namely $\pi = (3 + \sqrt{3})^{-1} =: \pi_0$ and $\pi=0.5$ which arise due to excess kurtosis and skewness respectively vanishing in those particular values of $\pi$. And thirdly, we observe that the curve goes to zero when approaching the point $\pi_0$ from the right. This counterintuitively suggests using only kurtosis even though $\kappa_1 \approx 0$ in the vicinity of $\pi_0$. However, a careful examination shows that when $\pi = \pi_0 + \epsilon$ for some small $\epsilon > 0$, the function $f$ indeed has a global minimum near zero but it also has $\text{lim}_{\alpha \rightarrow 0+} f(\alpha) = \infty$. Thus for practical purposes the global minimum might be too close to zero to be of any use.

%For example if $\pi = 0.22$ and $\mu = 5$, then $\hat{\alpha} \approx 0.26$ and $f(\hat{\alpha}) \approx 0.42$ and $f(1) \approx 0.94$, a relatively small difference in practice.

Hence, based on the above considerations, we thus suggest $\alpha = 0.8$ as a good, all-around weight for use in group separation of normal mixtures. This choice is further supported by the fact that it corresponds to the weights given to squared skewness and squared excess kurtosis in the classical Jarque-Bera test statistic for normality $(n/6) \hat{\gamma}^2 + (n/24)\hat{\kappa}^2$ \citep{jarque1987test} (under normality,  $(n/6) \hat{\gamma}^2$ and $(n/24)\hat{\kappa}^2$ are independent and both have a chi-squared distribution with one degree of freedom).
This corresponds  also to the effective value derived in \cite{jones1987projection}. %One could however argue that there is nothing special about the proportion $\alpha = 0.8$, as using the alternative weighting system discussed in Section \ref{sec:uni} the corresponding standardized weight is $\alpha_0 = 0.5$, implying that we merely advocate giving the same weight for standardized skewness and excess kurtosis, a very intuitive decision.

\subsection{Comparison of asymptotic variances in IC models}

Due to the affine equivariance of the estimates, it is sufficient to consider the behavior of $\hat{\textbf{W}}$ only in the case $\mathbf{\Omega}=\textbf{I}_p$. For all estimates, 
$\sqrt{n}vec(\hat{\textbf{W}}-\textbf{I}_p)\to_d N_{p^2}(\textbf{0}, \mathbf{\Upsilon})$ and the comparison should then be made using the asymptotic covariance matrix $\mathbf{\Upsilon}$. Then 
$tr(\mathbf{\Upsilon})$ is 
\[ \sum_{k=1}^p\sum_{l=1}^p ASV(\hat{w}_{kl}) = \sum_{k=1}^p ASV(\hat{w}_{kk}) +\sum_{k=1}^{p-1}\sum_{l=k+1}^p \left(ASV(\hat{w}_{kl}) + ASV(\hat{w}_{lk})\right) \]
where $\sum_{k=1}^p ASV(\hat{w}_{kk})$ is the same for all estimates. As in  \cite{miettinen2014fourth} we then use the values $ ASV(\hat{w}_{kl}) + ASV(\hat{w}_{lk})$ for the comparison
of the estimates for different choices of $k$th and $l$th marginal distributions. Surprisingly, for all deflation-based and symmetric projection pursuit estimates, this criterion value depends only on the $k$th and $l$th marginal distributions. For the estimates that use compound cumulant matrices, we use $ ASV(\hat{w}_{12}) + ASV(\hat{w}_{21})$ with $p=2$ and the same marginal distributions, as it is in fact the lower bound of $ ASV(\hat{w}_{kl}) + ASV(\hat{w}_{lk})$ if all other components are symmetric.

As the asymptotic variances of the symmetric projection pursuit approach and of the approach based on all cumulant matrices are the same (with adjusted weights), we in fact have three different methods to compare, namely, deflation-based PP, symmetric PP and the estimate based on  compound cumulant matrices. For each method we distinguish the versions using third cumulants only ($\alpha=1$), fourth cumulants only ($\alpha=0$),  and third and fourth cumulants with the weight $\alpha = 0.8$. The $k$th and $l$th marginal distributions are standardized versions of the exponential power distribution EP$(\alpha)$ or Gamma$(\alpha)$ with densities
\[
f(z)=\kappa_1 e^{-\kappa_2 |z|^{\alpha}} \ \ \mbox{or} \ \ f(z)=\kappa_1 z^{\alpha-1} e^{-\kappa_2 z}, \ z>0, 
\]
with $\kappa_1,\kappa_2>0$ and positive shape parameter $\alpha$. The asymptotic variances (and their lower bounds) then depend on the marginal distributions only through their shape parameters $\alpha$ 
For more details on the distributions in a similar study see \cite{miettinen2014fourth}. The values of $ASV(\hat{w}_{kl}) + ASV(\hat{w}_{lk})$ for different combinations of families and parameters are shown in Figures \ref{fig:sim_both} and \ref{fig:sim_gamma}. We do not report the results in cases where both components come from the symmetric exponential power family. In this case, the asymptotic variances of the estimates with $0<\alpha<1$ are the same as the asymptotic variance of the estimate with $\alpha=0$ and the results for $\alpha=0$ are already given in \cite{miettinen2014fourth}. In figures,  a darker shade indicates a larger value so that the performance of a particular method is at its best in the areas of lighter color.

From the contour plots we see that in the cases considered the performances of the estimates based on compound cumulant matrices are clearly the worst. One reason for this is, that none of them permit two sources having exactly the same distributions, causing the darker shades in the diagonals of Figure \ref{fig:sim_gamma} (see the Assumptions \ref{assu:distskew}, \ref{assu:distkurt} and \ref{assu:distboth} in Section \ref{sec:ICmodel}). It also seems that the symmetric projection pursuit (and the multiple cumulant method) gives the best performance, although the deflation-based methods do not come far behind.

Note, that the symmetry of exponential power distribution is evident in the upper two plots on the left-hand side of Figure \ref{fig:sim_both} where the sum of variances depends clearly only on the properties of the gamma distribution. The same two plots also showcase the fact that in the bivariate case when exactly one of the independent components is symmetric, both skewness-based projection pursuit methods have the same asymptotic behavior (as measured by the sum of off-diagonal asymptotic variances). The same would also hold for kurtosis, as can be verified by inspecting the results in Corollaries \ref{cor:sep_asymp} and \ref{cor:sym_asymp}.

\begin{figure}[t]
    \centering
    \includegraphics[width=1\textwidth]{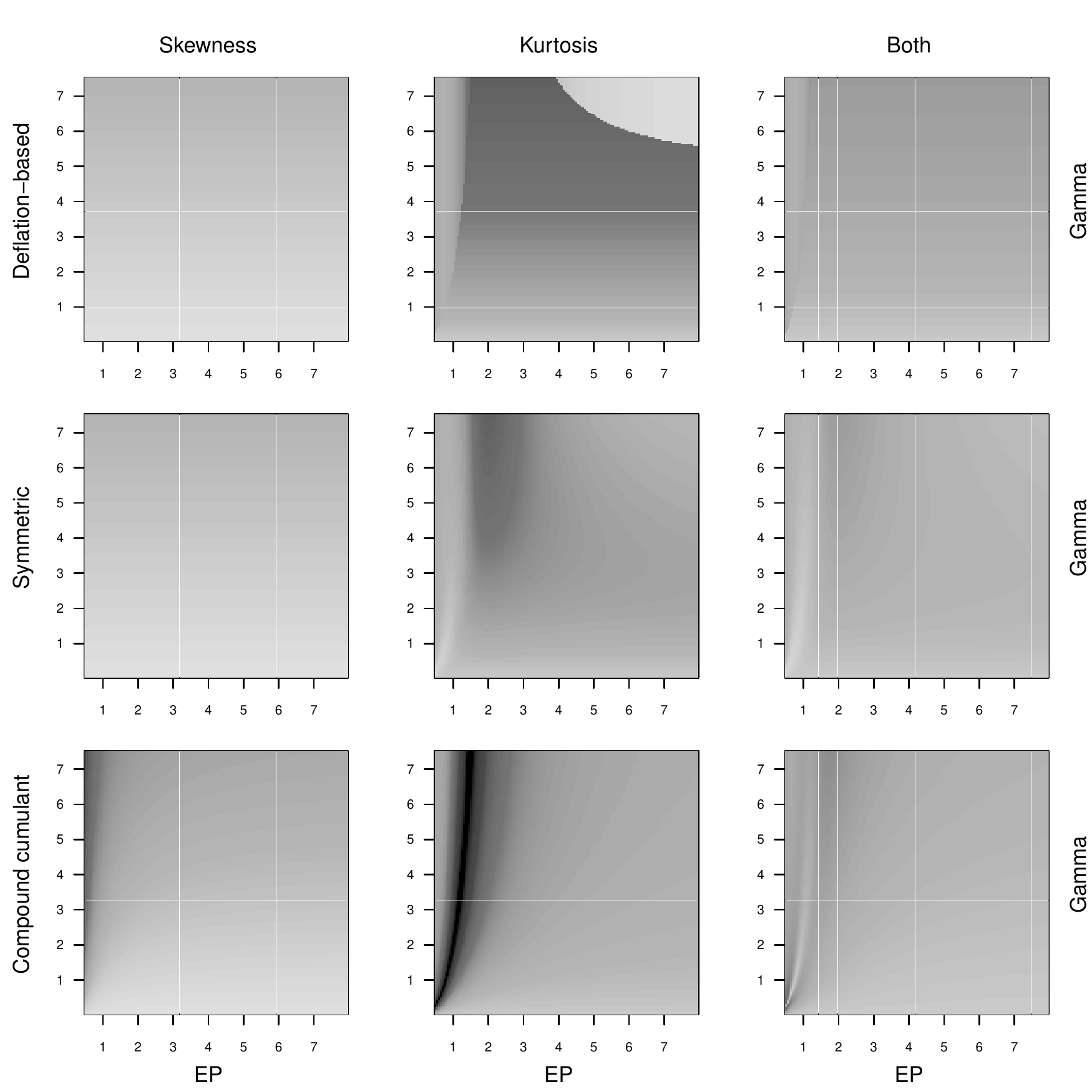}
    \caption{Contour plots of $ASV(\hat{w}_{12}) + ASV(\hat{w}_{21})$ for different combinations of methods and cumulants used when the $x$-axis independent component has an exponential power distribution and the $y$-axis independent component has a gamma distribution.}
    \label{fig:sim_both}
\end{figure}

\begin{figure}[t]
    \centering
    \includegraphics[width=1\textwidth]{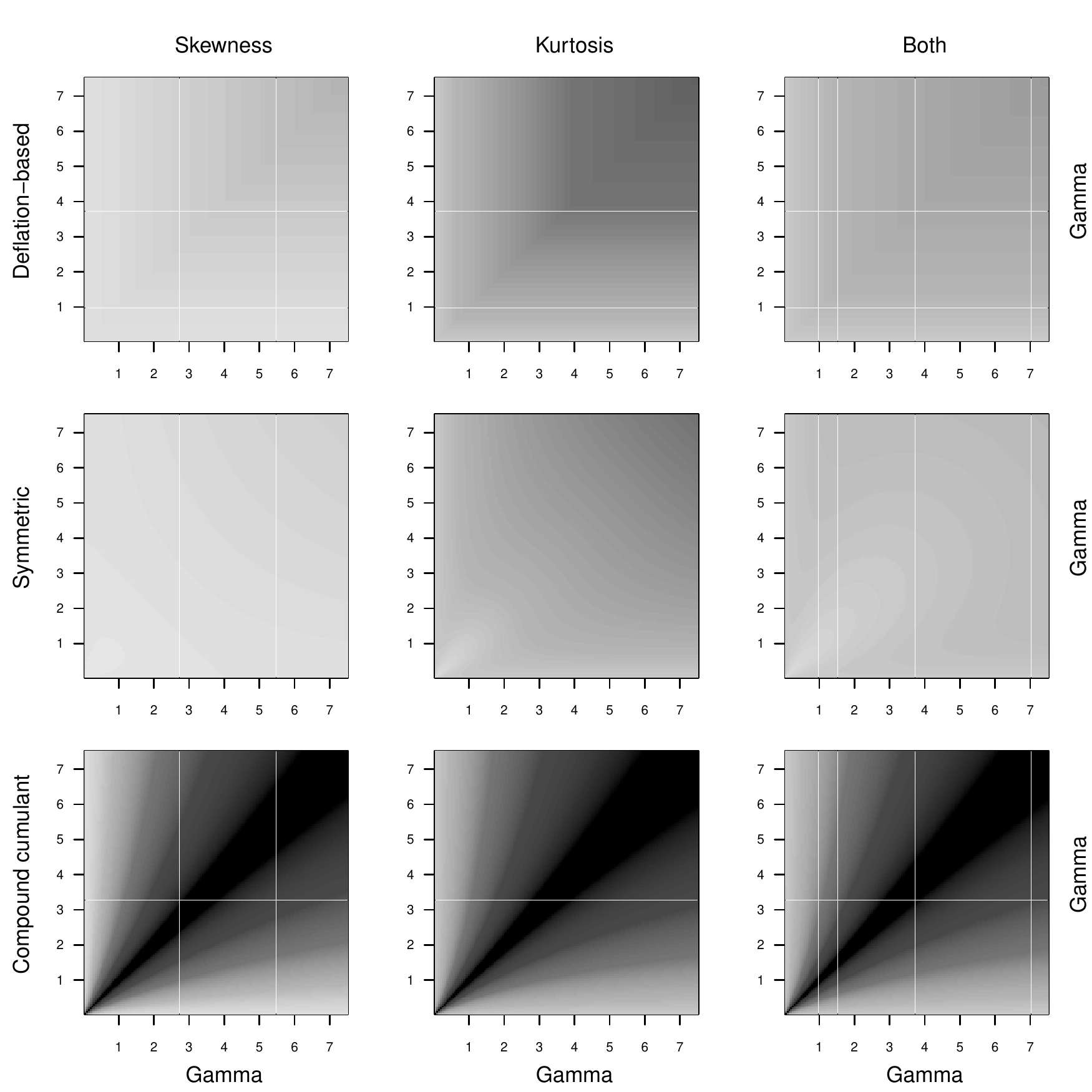}
    \caption{Contour plots of $ASV(\hat{w}_{12}) + ASV(\hat{w}_{21})$ for different combinations of methods and cumulants used when both independent components come from gamma distributions.}
    \label{fig:sim_gamma}
\end{figure}

\section{Discussion} \label{sec:conc}

In the previous sections, four different approaches for solving the independent component problem were thoroughly discussed. Each method was first precisely defined and then had its affine equivariance proven
 and estimating equations and algorithms provided, and finally the methods' asymptotic properties were derived. The main novelty in this paper is the combination of third and fourth cumulants in ICA where the weight given to skewness (or kurtosis) can be considered a tuning parameter. The special case of giving all weight to kurtosis
  yields then in the corresponding cases the deflation-based FastICA and the classic FOBI. Whereas the novel symmetric approach then gives perhaps a more natural version of the currently used symmetric FastICA approach.

The most surprising result here is the similar asymptotic behaviors of the symmetric projection pursuit and the method based on all cumulant matrices (including JADE).  Note that the squared symmetric projection pursuit is computationally much lighter than JADE, and could thus possibly replace the use of JADE in many applications. Following this discovery, a justified question to ask is whether moving from the absolute values to squares provides better results also in the general case of symmetric FastICA. This will be considered in a separate paper.
Another surprising result was that the compound cumulant approach needs special treatment to obtain affine equivariance when combing third and fourth compound cumulant matrices. Although the price to pay for this seems
relatively low as just a few stronger assumptions are needed. However as our comparisons indicate, this approach in general seems to be inferior to all other methods discussed here and its main advantage is its computational simplicity.

In the comparison section we established that all the methods can also be used successfully in cluster identification in the case of a multivariate normal mixture. Additionally, when using the projection pursuit methods for such a goal, a good rule of thumb for the choice of weights for squared skewness relative to squared kurtosis is to give 80\% of the weight to squared skewness, or alternatively, giving equal weights to standardized squared skewness and standardized squared excess kurtosis. This weighting then also coincides with the weighting used in the classical Jarque-Bera test of normality based on the same momentary quantities.

%Also the simulation results of Section \ref{sec:simu} gave further evidence on the superiority of using the squares.

Finally, it is interesting that although all the methods considered are defined very differently from each other, the corresponding expressions for the asymptotic variances in Corollaries \ref{cor:sep_asymp}, \ref{cor:sym_asymp}, \ref{cor:fobi_asymp} and \ref{cor:jade_asymp} exhibit pleasing symmetry. Based on this pattern one could even make a highly educated guess on what the asymptotic properties of the even higher moment versions of the methods would be (assuming that the methods actually exist).

\section{Acknowledgements}
Jari Miettinen kindly provided the code used for the comparison of the methods. This work was supported by the Academy of Finland (grant 268703).

\appendix

\setcounter{secnumdepth}{0}

%\begin{supplement}[id=suppA]
%\sname{Supplement A}
%\stitle{The proofs of Joint Use of Third and Fourth Cumulants in Independent Component Analysis}
%\slink[doi]{COMPLETED BY THE TYPESETTER}
%\sdatatype{.pdf}
%\sdescription{The proofs of select theorems and lemmas.}

\setcounter{aplemma}{0}

\makeatletter
\renewcommand{\theaplemma}{A\@arabic\c@aplemma}
\makeatother

\section{Appendix}\label{sec:proofs}

\begin{proof}[Proof of Theorem \ref{theo:sep_ineq}]
Note first, that under the assumptions of the independent component model the following two identities hold.
\[\gamma(\textbf{u}^T\textbf{z}) = \sum_{k=1}^p u_k^3\gamma_k \quad \text{and} \quad \kappa(\textbf{u}^T\textbf{z}) = \sum_{k=1}^p u_k^4\kappa_k. \]
Then, by using the Cauchy-Schwarz inequality and the fact that $u_k^2 \leq 1, \forall k$ we have
\begin{align*}
&\alpha_1 \gamma^2(\textbf{u}^T\textbf{z}) + \alpha_2 \kappa^2(\textbf{u}^T\textbf{z}) \\
\leq & \alpha_1\sum_{k=1}^p u_k^4 \gamma_k^2 + \alpha_2\sum_{k=1}^p u_k^6 \kappa_k^2 \\
\leq & \sum_{k=1}^p u_k^2 \underset{1 \leq l \leq p}{\text{max}}(\alpha_1 \gamma_l^2 + \alpha_2 \kappa_l^2),
\end{align*}
from which the result follows.
\end{proof}

\begin{proof}[Proof of Theorem \ref{theo:sep_asymp}]

We begin by proving the consistency of the estimator and due to the affine equivariance of the squared deflation-based projection pursuit functional $\textbf{W}$, we may without loss of generality restrict our attention to the case $\boldsymbol{\Omega} = \textbf{I}$ (this holds true for all the methods considered). Note then that the population and sample objective functions are of the forms
\[D(\textbf{u}) = \sum_{j=1}^J w_j \left( E\left[g_j(\textbf{u}^T \textbf{z})\right] \right)^2 \quad \text{and} \quad D_n(\textbf{u}) = \sum_{j=1}^J w_j \left( \frac{1}{n} \sum_{i=1}^n g_j(\textbf{u}^T \textbf{x}_{st,i}) \right)^2,\]
where $w_j$ are the weights given to the functions $g_j$. Consequently
\[\underset{\textbf{u}^T \textbf{u} = 1}{\text{sup}}|D(\textbf{u}) - D_n(\textbf{u})| \leq \sum_{j=1}^J w_j \underset{\textbf{u}^T \textbf{u} = 1}{\text{sup}} \left| \left( E\left[g_j(\textbf{u}^T \textbf{z})\right] \right)^2 - \left( \frac{1}{n} \sum_{i=1}^n g_j(\textbf{u}^T \textbf{x}_{st,i}) \right)^2 \right|. \]
The difference of squares then factorizes into form $(G_j + \hat{G}_j)(G_j - \hat{G}_j)$, where the first factor (for our choices of $g_j$) converges to finite constant due to the assumption on finiteness of moments and for the second factor we can use the uniform law of large numbers. As our choices for the functions $g_j$ are continuous and the set $\{ \textbf{u} \in \mathbb{R}^p : \textbf{u}^T \textbf{u} = 1 \}$ is compact, we then have $\text{sup}_{\textbf{u}^T \textbf{u} = 1} | D(\textbf{u}) - D_n(\textbf{u}) | \rightarrow_P 0$.

Now $D(\textbf{u})$ has the unique (up to sign) maximizer $\textbf{e}_1$, and applying the technique used in the proofs of \citet{miettinen2014deflation}, the above uniform convergence in probability implies the consistency of step 1 (up to sign), $\mathbb{P}(\|\hat{\textbf{w}}_1 - \textbf{e}_1 \| < \epsilon) \rightarrow_P 0$.

For the convergence of step 2, we follow in the vein of \citet{miettinen2014deflation} and move to the orthogonal complement $\hat{\textbf{u}}_1^\perp$ of the span of $\hat{\textbf{u}}_1$ and consider the functions $D_2(\textbf{v}) := D(\textbf{E}\textbf{v})$ and $D_{2,n}(\textbf{v}) := D_n(\hat{\textbf{E}}\textbf{v})$, where $\textbf{v} \in \mathbb{R}^{p-1}$, $\textbf{E} = (\textbf{e}_2, ..., \textbf{e}_p) \in \mathbb{R}^{p \times (p-1)}$ and $\hat{\textbf{E}}$ is chosen as the closest matrix to $\textbf{E}$ with respect to matrix norm such that the matrix $(\hat{\textbf{u}_1}, \hat{\textbf{E}})$ is orthogonal. Note that this is not restricting as $\textbf{E}$ and $\hat{\textbf{E}}$ are bases of $\textbf{e}_1^\perp$ and $\hat{\textbf{u}}_1^\perp$, respectively. The consistency of $\hat{\textbf{u}}_1$ also implies $\hat{\textbf{E}} \rightarrow_P \textbf{E}$.

Similar reasoning as used above with $D(\textbf{u})$ and $D_n(\textbf{u})$ in conjunction with the following convergence implied by the results in \citet{randles1982asymptotic} and the finiteness of moments and differentiability of our choice of functions $g_j$,
\[ \frac{1}{n}\sum_{i=1}^n g_j(\textbf{v}^T \hat{\textbf{E}}{}^T \textbf{x}_{st,i}) \rightarrow_P E\left[g_j(\textbf{v}^T\textbf{E}^T \textbf{z})\right], \]
can be used to prove the convergence, $\text{sup}_{\textbf{v}^T \textbf{v} = 1} | D_2(\textbf{v}) - D_{2,n}(\textbf{v}) | \rightarrow_P 0$. Observing then that $D(\textbf{v})$ has the unique (up to sign) maximizer $\textbf{e}_1$, the arguments used for $\hat{\textbf{u}}_1$ then show that $\hat{\textbf{v}} \rightarrow_P \textbf{e}_1$ and consequently $\hat{\textbf{u}}_2 = \hat{\textbf{E}} \hat{\textbf{v}} \rightarrow_P \textbf{e}_2$. Using similar constructions for $k = 3,...,p-1$ we get the consistency of the estimator up to sign-change, that is $\hat{\textbf{W}} \rightarrow_P \textbf{I}_p$.

For the asymptotic behavior of $\hat{\textbf{W}}$ we then consider the diagonal and off-diagonal elements of $\hat{\textbf{W}}$ separately, and starting with the diagonal elements we first establish the following Lemma.

\begin{aplemma}\label{lem:app_diag}
Assume that $\hat{\textbf{W}} = \hat{\textbf{U}} \hat{\textbf{S}}{}^{-1/2}$, where $\sqrt{n}(\hat{\textbf{S}} - \textbf{I}_p) = \mathcal{O}_P(1)$, $\sqrt{n}(\hat{\textbf{U}} - \textbf{I}_p) = \mathcal{O}_P(1)$ and $\hat{\textbf{U}} \in \mathcal{U}$. Then the following three hold.
\begin{alignat*}{2}
\sqrt{n} (\hat{w}_{kk} - 1) &= -\frac{1}{2} \sqrt{n} (\hat{s}_{kk} - 1) + o_P(1), \quad & k&=1,...,p, \\
\sqrt{n} \hat{w}_{kl} + \sqrt{n} \hat{w}_{lk} &= -\sqrt{n} \hat{s}_{kl} + o_P(1), \quad & k \neq l&=1,...,p, \\
\sqrt{n} \hat{w}_{kl} &= \sqrt{n} \hat{u}_{kl} -\frac{1}{2}\sqrt{n} \hat{s}_{kl} + o_P(1), \quad & k \neq l&=1,...,p.
\end{alignat*}
\end{aplemma}
To prove Lemma \ref{lem:app_diag} consider first the following identity.
\begin{align*}
\textbf{O}_p &= \hat{\textbf{S}}{}^{-1/2} \sqrt{n} ( \hat{\textbf{S}} - \textbf{I}_p) \hat{\textbf{S}}{}^{-1/2} + \hat{\textbf{S}}{}^{-1/2} \sqrt{n} (\hat{\textbf{S}}{}^{-1/2} - \textbf{I}_p) + \sqrt{n}( \hat{\textbf{S}}{}^{-1/2} - \textbf{I}_p) \\
&= \sqrt{n} ( \hat{\textbf{S}} - \textbf{I}_p) + 2 \sqrt{n}( \hat{\textbf{S}}{}^{-1/2} - \textbf{I}_p) + o_P(1).
\end{align*}
For the second equality above, note that $\sqrt{n} ( \hat{\textbf{S}} - \textbf{I}_p) = - \sqrt{n} (\hat{\textbf{S}}{}^{-1/2} - \textbf{I}_p) (\hat{\textbf{S}} + \hat{\textbf{S}}{}^{1/2})$ which implies that $\sqrt{n} (\hat{\textbf{S}}{}^{-1/2} - \textbf{I}_p)$ is bounded in probability, thus allowing us to conclude the identity $\sqrt{n}(\hat{\textbf{S}}{}^{-1/2} - \textbf{I}_p) = -(1/2)\sqrt{n}(\hat{\textbf{S}} - \textbf{I}_p) + o_P(1)$.

Using similar techniques one can prove that $\sqrt{n}(\hat{\textbf{U}}{}^T - \textbf{I}_p) = -\sqrt{n}(\hat{\textbf{U}} - \textbf{I}_p) + o_P(1)$ and $\sqrt{n}(\hat{\textbf{W}} - \textbf{I}_p) = \sqrt{n}(\hat{\textbf{U}} - \textbf{I}_p) + \sqrt{n}(\hat{\textbf{S}}{}^{-1/2} - \textbf{I}_p) + o_P(1)$. As a consequence of these we then get the third claim of the lemma.

Consider then the sum of $\hat{\textbf{W}}$ with its transpose $\hat{\textbf{W}} + \hat{\textbf{W}}{}^T = \hat{\textbf{U}} \hat{\textbf{S}}{}^{-1/2} + \hat{\textbf{S}}{}^{-1/2} \hat{\textbf{U}}{}^T$.
\begin{align*}
\sqrt{n}(\hat{\textbf{W}} + \hat{\textbf{W}}{}^T - 2\textbf{I}_p) &= \sqrt{n}(\hat{\textbf{U}} - \textbf{I}_p) \hat{\textbf{S}}{}^{-1/2} + \sqrt{n}( \hat{\textbf{S}}{}^{-1/2} - \textbf{I}_p) \\
& \quad + \sqrt{n}( \hat{\textbf{S}}{}^{-1/2} - \textbf{I}_p) \hat{\textbf{U}}{}^T + \sqrt{n}(\hat{\textbf{U}}{}^T - \textbf{I}_p) + o_P(1) \\
& = - \sqrt{n}( \hat{\textbf{S}} - \textbf{I}_p) + o_P(1),
\end{align*}
from which the first two claims follow.

For the asymptotic behavior of the off-diagonal elements we require in the current proof and the proof of Theorem \ref{theo:sym_asymp} the following estimators.
\begin{alignat*}{4}
& \hat{h}_{3k} = \frac{1}{n} \sum_{i=1}^n (\hat{\textbf{w}}_k^T \tilde{\textbf{z}}_i)^3, \qquad & \hat{\textbf{T}}_{3k} &= \frac{1}{n} \sum_{i=1}^n (\hat{\textbf{w}}_k^T \tilde{\textbf{z}}_i)^2 \tilde{\textbf{z}}_i, \\
& \hat{h}_{4k} = \frac{1}{n} \sum_{i=1}^n (\hat{\textbf{w}}_k^T \tilde{\textbf{z}}_i)^4 - 3, \qquad & \hat{\textbf{T}}_{4k} &= \frac{1}{n} \sum_{i=1}^n (\hat{\textbf{w}}_k^T \tilde{\textbf{z}}_i)^3 \tilde{\textbf{z}}_i,
\end{alignat*}
satisfying $\hat{h}_{3k} \rightarrow_P \gamma_k, \hat{h}_{4k} \rightarrow_P \kappa_k, \hat{\textbf{T}}_{3k} \rightarrow_P \gamma_k \textbf{e}_k$ and $\hat{\textbf{T}}_{4k} \rightarrow_P \beta_k \textbf{e}_k$. Note then, that in terms of $\hat{\textbf{W}} = (\hat{\textbf{w}}_1,...,\hat{\textbf{w}}_p)^T = \hat{\textbf{U}} \hat{\textbf{S}}{}^{-1/2}$ the estimating equations have the form
\[\hat{\boldsymbol{\Gamma}}_k = \hat{\textbf{S}} (\sum_{j=1}^k \hat{\textbf{w}}_j \hat{\textbf{w}}_j^T) \hat{\boldsymbol{\Gamma}}_k, \]
where $\hat{\boldsymbol{\Gamma}}_k = 3 \alpha \hat{h}_{3k} \hat{\textbf{T}}_{3k} + 4 (1 - \alpha) \hat{h}_{4k}  \hat{\textbf{T}}_{4k}$. Then, using Equation (5) from \cite{nordhausen2011deflation} we get the identity
\begin{align*}
\textbf{J}_k \sqrt{n} (\hat{\boldsymbol{\Gamma}}_k - \boldsymbol{\Gamma}_k \textbf{e}_k) =& \boldsymbol{\Gamma}_k [ \sqrt{n} (\hat{\textbf{S}} - \textbf{I}_p) \textbf{e}_k + \sum_{j=1}^k \textbf{e}_j \textbf{e}_k^T \sqrt{n} ( \hat{\textbf{w}}_j - \textbf{e}_j) \numberthis \label{eq:app_nord5} \\
+& \sqrt{n}(\hat{\textbf{w}}_k - \textbf{e}_k) ] + o_P(1),
\end{align*}
where $\textbf{J}_k = \sum_{j>k} \textbf{e}_j \textbf{e}_j^T$ and $\boldsymbol{\Gamma}_k = 3 \alpha \gamma_k^2 + 4 (1 - \alpha) \kappa_k \beta_k$. Next, using Equation (3) from \cite{nordhausen2011deflation} separately for $\hat{\textbf{T}}_{3k}$ and $\hat{\textbf{T}}_{4k}$ gives the following two identities.
\begin{align*}
\sqrt{n} (\hat{\textbf{T}}_{3k} - \gamma_k \textbf{e}_k) &= \sqrt{n} \hat{\textbf{r}}_k - 2 \textbf{e}_k \textbf{e}_k^T \sqrt{n} \bar{\textbf{z}} + 2 \gamma_k \textbf{E}^{kk} \sqrt{n} (\hat{\textbf{w}}_k - \textbf{e}_k) + o_P(1), \numberthis \label{eq:app_t3k} \\
\sqrt{n} (\hat{\textbf{T}}_{4k} - \beta_k \textbf{e}_k) &= \sqrt{n} \hat{\textbf{q}}_k - 3 \gamma_k \textbf{e}_k \textbf{e}_k^T \sqrt{n} \bar{\textbf{z}} \numberthis \label{eq:app_t4k} \\
&+ 3 (\textbf{I}_p + (\beta_k - 1) \textbf{E}^{kk}) \sqrt{n} (\hat{\textbf{w}}_k - \textbf{e}_k) + o_P(1),
\end{align*}
where $\hat{\textbf{r}}_k = (1/n) \sum_{i=1}^n (z_{ik}^2 - 1) \textbf{z}_i$ and $\hat{\textbf{q}}_k = (1/n) \sum_{i=1}^n (z_{ik}^3 - \gamma_k) \textbf{z}_i$. Using Equations \eqref{eq:app_t3k} and \eqref{eq:app_t4k} together with the fact that $\sqrt{n} (\hat{h}_{3k} \hat{\textbf{T}}_{3k} - \gamma_k^2 \textbf{e}_k) = \gamma_k \sqrt{n} ( \hat{\textbf{T}}_{3k} - \gamma_k \textbf{e}_k) + \gamma_k \sqrt{n} (\hat{h}_{3k} - \gamma_k) \textbf{e}_k + o_P(1)$ (and the analogy for $\hat{\textbf{T}}_{4k}$) we get an alternative expression for $\sqrt{n} (\hat{\boldsymbol{\Gamma}}_k - \boldsymbol{\Gamma}_k \textbf{e}_k)$ which can be substituted into Equation \eqref{eq:app_nord5}. Inspecting the result element-wise then yields the following two equations from which the asymptotic result follows.
\begin{align*}
0 &= \sqrt{n} \hat{s}_{lk} + \sqrt{n} \hat{w}_{lk} + \sqrt{n} \hat{w}_{kl} + o_P(1), \quad l < k,
\end{align*}
and
\begin{align*}
\quad & 3 \alpha \gamma_k \sqrt{n} \hat{r}_{kl} + 4 (1 - \alpha) \kappa_k ( \sqrt{n} \hat{q}_{kl} + 3 \sqrt{n} \hat{w}_{kl})\\
= \quad & (3 \alpha \gamma_k^2 + 4 (1 - \alpha) \kappa_k \beta_k)(\sqrt{n} \hat{s}_{lk} + \sqrt{n} \hat{w}_{kl}) + o_P(1), \quad l > k.
\end{align*}
\end{proof}

\begin{proof}[Proof of Theorem \ref{theo:sym_ineq}]
For the proof we require the following Lemma.
\begin{aplemma}\label{lem:app_ineq}
Let a $p \times p$ matrix $\textbf{U} \in \mathcal{U}$, $\textbf{b} \in \mathbb{R}^p$ and $r \in \mathbb{N}, \, r \geq 2$. Then
\[\sum_{i=1}^p \sum_{k=1}^p \sum_{l=1}^p u_{ik}^r u_{il}^r b_k b_l \leq \sum_{k=1}^p b_k^2. \]
\end{aplemma}
To prove Lemma \ref{lem:app_ineq} we first utilize the Cauchy-Schwarz inequality.
\[\sum_{i=1}^p \sum_{k=1}^p \sum_{l=1}^p u_{ik}^r u_{il}^r b_k b_l = \sum_{i=1}^p \left( \sum_{k=1}^p (u_{ik}) (u_{ik}^{r-1} b_k) \right)^2 \leq \sum_{i=1}^p \sum_{k=1}^p u_{ik}^{2r-2} b_k^2. \]
Then observing that $u_{ik}^{2r-2} = u_{ik}^2 u_{ik}^{2r-4} \leq u_{ik}^2$ gives the desired result.

The inequalities of Theorem \ref{theo:sym_ineq} then easily follow by first expanding the left-hand sides under the assumptions of the independent component model in \eqref{eq:icm_icmodel} to yield
\[
\sum_{i=1}^p \sum_{k=1}^p \sum_{l=1}^p u_{ik}^3 u_{il}^3 \gamma_k \gamma_l \quad \text{and} \quad
\sum_{i=1}^p \sum_{k=1}^p \sum_{l=1}^p u_{ik}^4 u_{il}^4 \kappa_k \kappa_l.
\]
Then, for both cases, an application of Lemma \ref{lem:app_ineq} gives the desired result.

\end{proof}

\begin{proof}[Proof of Lemma \ref{lem:estimeq}]
The matrix form $\textbf{U} \textbf{T}^T = \textbf{T} \textbf{U}^T$ of the estimating equations follows easily by element-wise inspection. This yields further $(\textbf{T}^T \textbf{U})^2 = \textbf{T}^T \textbf{T}$ from which the result follows by first taking the symmetric square root of both sides.
\end{proof}

\begin{proof}[Proof of Theorem \ref{theo:sym_asymp}]
For the consistency, the uniform convergence in probability of the sample objective function to the population one follows easily from the proof of Theorem \ref{theo:sep_asymp} as the objective functions $D(\textbf{U})$ and $D_n(\textbf{U})$ are now just sums of the individual objective functions of the squared deflation-based projection pursuit. The desired result $\mathbb{P}( \| \hat{\textbf{W}} - \textbf{I}_p \| < \epsilon) \rightarrow 1, \, \forall \epsilon > 0$, is then proven similarly as in \citet{miettinen2014separation}

For the asymptotic behavior, Lemma \ref{lem:app_diag} takes care of the diagonal elements so we will only need to consider the off-diagonal elements. The sample versions of the estimating equations for $k,l = 1,...,p$ are
\[3 \alpha \hat{h}_{3k} \hat{\textbf{w}}_l^T \hat{\textbf{T}}_{3k} + 4 (1 - \alpha) \hat{h}_{4k} \hat{\textbf{w}}_l^T \hat{\textbf{T}}_{4k} = 3 \alpha \hat{h}_{3l} \hat{\textbf{w}}_k^T \hat{\textbf{T}}_{3l} + 4 (1 - \alpha) \hat{h}_{4l} \hat{\textbf{w}}_k^T \hat{\textbf{T}}_{4l}, \numberthis \label{eq:sym_ee}  \]
where $\hat{h}_{3k}$, $\hat{h}_{4k}$, $\hat{\textbf{T}}_{3k}$ and $\hat{\textbf{T}}_{4k}$ are as in the proof of Theorem \ref{theo:sep_asymp}. With a approach similar to the one used in the proof of Theorem 6 in \cite{miettinen2014fourth} we have that
\[\sqrt{n} \hat{h}_{3l} \hat{\textbf{w}}_k^T \hat{\textbf{T}}_{3l} = \gamma_l \sqrt{n} (\hat{\textbf{w}}_k - \textbf{e}_k)^T \gamma_l \textbf{e}_l + \gamma_l \textbf{e}_k^T \sqrt{n} (\hat{\textbf{T}}_{3l} - \gamma_l \textbf{e}_l) + o_P(1), \]
and
\[\sqrt{n} \hat{h}_{4l} \hat{\textbf{w}}_k^T \hat{\textbf{T}}_{4l}  = \kappa_l \sqrt{n} (\hat{\textbf{w}}_k - \textbf{e}_k)^T \beta_l \textbf{e}_l + \kappa_l \textbf{e}_k^T \sqrt{n} (\hat{\textbf{T}}_{4l} - \beta_l \textbf{e}_l) + o_P(1). \]

Substituting Equations \eqref{eq:app_t3k} and \eqref{eq:app_t4k} into the above expansions and then using the symmetry of the estimating equations in \eqref{eq:sym_ee} gives the following identity.

\begin{align*}
& 3 \alpha (\gamma_k^2 \sqrt{n} \hat{w}_{lk} + \gamma_k \sqrt{n} \hat{r}_{kl}) + 4 (1 - \alpha) (\beta_k \kappa_k \sqrt{n} \hat{w}_{lk} + \kappa_k \sqrt{n} \hat{q}_{kl} + 3 \kappa_k \sqrt{n} \hat{w}_{kl}) \\
= & 3 \alpha (\gamma_l^2 \sqrt{n} \hat{w}_{kl} + \gamma_l \sqrt{n} \hat{r}_{lk}) + 4 (1 - \alpha) (\beta_l \kappa_l \sqrt{n} \hat{w}_{kl} + \kappa_l \sqrt{n} \hat{q}_{lk} + 3 \kappa_l \sqrt{n} \hat{w}_{lk}) + o_P(1),
\end{align*}
from which the asymptotic result then follows using the second identity of Lemma \ref{lem:app_diag}.
\end{proof}

\begin{proof}[Proof of Theorem \ref{theo:diag}]
Evaluating $\textbf{C}^{3i}$ at $\textbf{x}_{st} = \textbf{U}^T\textbf{z}$ yields
\begin{align*}
\textbf{C}^{3i}(\textbf{x}_{st}) &= \textbf{U}^T E \left[ (\textbf{z}^T \textbf{U} \textbf{e}_i) \textbf{z} \textbf{z}^T \right] \textbf{U} = \textbf{U}^T \left( \sum_{k=1}^p u_{ki} \textbf{C}^{3k}(\textbf{z}) \right) \textbf{U},
\end{align*}
which proves the theorem for $\textbf{C}^3(\textbf{x}_{st})$ also. Similarly, after some simplification, we have for fourth joint cumulants
\begin{align*}
\textbf{C}^{4ij}(\textbf{x}_{st}) &= \textbf{U}^T \left( \sum_{k,l}^p u_{ki} u_{lj} \textbf{B}^{kl}(\textbf{z}) - \delta_{ij} \textbf{I}_p - \textbf{u}_i \textbf{u}_j^T - \textbf{u}_j \textbf{u}_i^T \right) \textbf{U} \\
&= \textbf{U}^T \left( \sum_{k=1}^p u_{ki} u_{kj} \kappa_k \textbf{E}^{kk} \right) \textbf{U},
\end{align*}
where $\textbf{u}_k$ are columns of $\textbf{U}$. This then gives the result for $\textbf{C}^4(\textbf{x}_{st})$ also.
\end{proof}

\begin{proof}[Proof of Corollary \ref{cor:fobi_xor}]
Observe first, that without loss of generality, we may in both cases assume that the non-zero weight is equal to 1. Starting with the third cumulants, we have under the independent component model $\textbf{C}^{3i}(\textbf{z}) = \gamma_i \textbf{E}^{ii}$. Then from the proof of Theorem \ref{theo:diag} we have that
\[\textbf{C}^{3i}(\textbf{x}_{st}^*) = \textbf{C}^{3i}(\textbf{U}^* \textbf{U}^T \textbf{z}) = \textbf{U}^* \textbf{U}^T \textbf{D}_i \textbf{U} \textbf{U}^{*T},\]
where the matrices $\textbf{D}_i$, $i=1,...,p$ are diagonal. This in turn implies that the matrix $\textbf{C}^3(\textbf{x}_{st}^*)$ has
\begin{align*}
\textbf{C}^3(\textbf{x}_{st}^*) &= \sum_{i=1}^p \textbf{C}^{3i}(\textbf{x}_{st}^*) = \textbf{U}^* \textbf{U}^T \left( \sum_{i=1}^p \textbf{D}_i \right) \textbf{U} \textbf{U}^{*T},
\end{align*}
the last line of which is the eigendecomposition (diagonalization) of the matrix $\textbf{C}^3(\textbf{x}_{st}^*)$. Thus choosing $\alpha = 1$ in the optimization problem of Definition \ref{def:fobi_func} leads to this same diagonalization and gives the transformation $\textbf{x}_{st}^* \mapsto \textbf{U} \textbf{U}^{*T} \textbf{x}_{st}^* = \textbf{z}$.

For the corresponding proof for fourth moments and the FOBI-matrix, $E\left[ \textbf{x}{}^*_{st} \textbf{x}{}^{*T}_{st} \textbf{x}{}^*_{st} \textbf{x}{}^{*T}_{st} \right]$, denote $\textbf{U}^* \textbf{U}^T = \textbf{V} \in \mathcal{U}$ and see for example \cite{miettinen2014fourth}. The same result for $\textbf{C}^4(\textbf{x}{}^*_{st}) = E\left[ \textbf{x}{}^*_{st} \textbf{x}{}^{*T}_{st} \textbf{x}{}^*_{st} \textbf{x}{}^{*T}_{st} \right] - (p+2) \textbf{I}_p$ then instantly follows
\end{proof}

%\begin{proof}[Proof of Lemma \ref{lem:fobi_ae}]
%We prove the affine equivariance of the method by showing that the optimization problem in Definition \ref{def:fobi_func} is invariant under mappings $\textbf{x}_{st} \mapsto \textbf{V}\textbf{x}_{st}$, where $\textbf{V}$ is an orthogonal matrix (the mapping corresponds to an affine transformation for the actual observations, see Section \ref{sec:nota}) and $\textbf{x}_{st} = \textbf{U}^T\textbf{z}$.

%Based on the proof of Corollary \ref{cor:fobi_xor}, we first establish that
%\[\textbf{C}^3(\textbf{V} \textbf{x}_{st}) = \textbf{V} \textbf{U}^T \left( \sum_{i=1}^p \sum_{k=1}^p (\textbf{U}\textbf{V}^T)_{ki} \gamma_k \textbf{E}^{kk} \right) \textbf{U} \textbf{V}^T  ,\]
%where the parenthesized part is diagonal. Similarly, using the alternative form $\textbf{C}^4(\textbf{x}_{st}) = E(\textbf{x}_{st} \textbf{x}_{st}^T \textbf{x}_{st} \textbf{x}_{st}^T) - (p+2) \textbf{I}_p$, we have
%\[\textbf{C}^4(\textbf{V} \textbf{x}_{st}) = \textbf{V} \textbf{U}^T \textbf{C}^4(\textbf{z}) \textbf{U} \textbf{V}^T,\]
%where the matrix $\textbf{C}^4(\textbf{z})$ is diagonal.

%Thus for all weights $\alpha \in [0,1]$ and $\textbf{V} \in \mathcal{U}$ the matrix $\textbf{U}\textbf{V}^T$ is the solution to the optimization problem in Definition \ref{def:fobi_func} applied to the transformed observations $\textbf{V}\textbf{x}_{st}$ (and yields the desired transformation  $\textbf{V} \textbf{x}_{st} \mapsto (\textbf{U} \textbf{V}^T) \textbf{V}\textbf{x}_{st} = \textbf{z}$).
%\end{proof}
%\hrule

\begin{proof}[Proof of Theorem \ref{theo:fobi_asymp}]
We begin by proving the consistency of the estimator. Again we first need to show that the sample objective function converges uniformly in probability to the corresponding population statistic over $\mathcal{U}$. For both the compound cumulant and multiple cumulant methods the objective functions are of the form
\[D(\textbf{U}) = \sum_{j=1}^J w_j \sum_{d=1}^p (\textbf{u}^T_d\textbf{M}_j\textbf{u}_d)^2 \quad \text{and} \quad D_n(\textbf{U}) = \sum_{j=1}^J w_j \sum_{d=1}^p (\textbf{u}^T_d\hat{\textbf{M}}_j\textbf{u}_d)^2, \]
where $\{\textbf{M}_j\}_{j=1}^J$ and $\{\hat{\textbf{M}}_j\}_{j=1}^J$ are the sets of matrices to be diagonalized and their sample versions, and $w_j$ are their respective positive weights. It is thus sufficient to consider individual supremums of the form
\[S(\textbf{M}, \hat{\textbf{M}}) = \underset{\textbf{u}^T\textbf{u} = 1}{\text{sup}}|(\textbf{u}^T\textbf{M}\textbf{u})^2 - (\textbf{u}^T\hat{\textbf{M}}\textbf{u})^2|,\]
where $\textbf{M}$ and $\hat{\textbf{M}}$ are the population and sample version of an arbitrary matrix to be diagonalized and thus satisfy $\hat{\textbf{M}} \rightarrow_P \textbf{M}$. The assumptions on the finiteness of moments further ensures that $\| \textbf{M} + \hat{\textbf{M}} \|$ converges in probability to some finite constant and we thus have
\begin{align*}
S(\textbf{M}, \hat{\textbf{M}}) &= \underset{\textbf{u}^T\textbf{u} = 1}{\text{sup}} \left( |\textbf{u}^T (\textbf{M} + \hat{\textbf{M}}) \textbf{u}| \cdot |\textbf{u}^T (\textbf{M} - \hat{\textbf{M}}) \textbf{u}| \right) \\
&\leq \underset{\textbf{u}^T\textbf{u} = 1}{\text{sup}} \left( \| \textbf{u}^T \| \| \textbf{M} + \hat{\textbf{M}} \| \| \textbf{u} \| \right) \underset{\textbf{u}^T\textbf{u} = 1}{\text{sup}} \left( \| \textbf{u}^T \| \| \textbf{M} - \hat{\textbf{M}} \| \| \textbf{u} \| \right)\\
&= \| \textbf{M} + \hat{\textbf{M}} \| \| \textbf{M} - \hat{\textbf{M}} \| \\
&\rightarrow_P 0.
\end{align*}

Using then the obtained result, $\text{sup}_{\textbf{U} \in \mathcal{U}}| D(\textbf{U}) - D_n(\textbf{U}) | \rightarrow_P 0$, the consistency of the estimator, that is $\mathbb{P}( \| \hat{\textbf{W}} - \textbf{I}_p \| < \epsilon) \rightarrow 1, \, \forall \epsilon > 0$, is proven similarly as in \citet{miettinen2014separation}.

Next, concerning the asymptotic behavior of $\hat{\textbf{W}}$, we again without loss of generality assume $\boldsymbol{\Omega} = \textbf{I}_p$, and then for diagonal elements it again suffices to use Lemma \ref{lem:app_diag}. To find the asymptotic behavior of the off-diagonal elements of $\hat{\textbf{W}}$ we in turn utilize the following lemma from the supplementary material of \cite{miettinen2014fourth}.

\begin{aplemma}\label{lem:app_multi}
Assume that $\hat{\textbf{S}}_k, k=1,...,K$ are $p \times p$ matrices such that $\sqrt{n}(\hat{\textbf{S}}_k - \boldsymbol{\Lambda}_k)$ are asymptotically normal with mean zero and $\boldsymbol{\Lambda}_k = diag(\lambda_{k1},...,\lambda_{kp})$. Let $\hat{\textbf{U}} = (\hat{\textbf{u}}_1,...,\hat{\textbf{u}}_p)^T$ be the orthogonal matrix that maximizes
\[\sum_{k=1}^K \| diag(\hat{\textbf{U}} \hat{\textbf{S}}_k \hat{\textbf{U}}{}^T) \|^2. \]
Then
\[\sqrt{n} \hat{u}_{ij} = \frac{\sum_{k=1}^K(\lambda_{ki} - \lambda_{kj})\sqrt{n} [\hat{\textbf{S}}_k]_{ij}}{\sum_{k=1}^K(\lambda_{ki} - \lambda_{kj})^2} + o_P(1), \quad i \neq j.\]
\end{aplemma}

Due to the weighting, the matrices we diagonalize are in terms of Lemma \ref{lem:app_multi} now actually $\sqrt{\alpha} \hat{\textbf{C}}{}^3$ and $\sqrt{1 - \alpha} \hat{\textbf{C}}{}^4$. Write then $\hat{\textbf{V}} = \hat{\textbf{U}}{}^* \hat{\textbf{S}}{}^{-1/2}$, so the estimated unmixing matrix has the form $\hat{\textbf{W}} =  \hat{\textbf{U}} \hat{\textbf{V}}$ and because  $\sqrt{n} (\hat{\textbf{U}}{}^* - \textbf{I}_p) = \mathcal{O}_P(1)$, also $\sqrt{n} (\hat{\textbf{V}} - \textbf{I}_p) = \mathcal{O}_P(1)$ holds.  Based on the orthogonality of $\hat{\textbf{U}}{}^*$ and the results of Lemma \ref{lem:app_diag}, we then have the following two equalities for $k\neq l=1,...,p$.
\begin{align*}
\sqrt{n} \hat{w}_{kl} &= \sqrt{n} \hat{u}_{kl} + \sqrt{n} \hat{v}_{kl} + o_P(1), \\
-\sqrt{n} \hat{s}_{kl} &= \sqrt{n} \hat{v}_{kl} + \sqrt{n} \hat{v}_{lk} + o_P(1)
\end{align*}

Using then Lemmas \ref{lem:app_diag} and \ref{lem:app_multi}, the above and the fact that $\textbf{C}^3(\textbf{z})$ and $\textbf{C}^4(\textbf{z})$ are diagonal we have for the compound cumulant method
\[\sqrt{n}\hat{w}_{kl} = \frac{\alpha \delta_{1kl}\sqrt{n}\left(\hat{\textbf{C}}{}^3_{kl} + \delta_{1kl}\hat{v}_{kl}\right) + (1 - \alpha) \delta_{2kl}\sqrt{n} \left(\hat{\textbf{C}}{}^4_{kl} + \delta_{2kl}\hat{v}_{kl}\right)}{\alpha \delta_{1kl}^2 + (1 - \alpha) \delta_{2kl}^2} + o_P(1), \]
where $\delta_{1kl} = (\gamma_k - \gamma_l)$ and $\delta_{2kl} = (\kappa_k - \kappa_l)$. By slightly modifying the proof of Theorem 8 in \cite{miettinen2014fourth} we can get the behavior of the FOBI-matrix $\hat{\textbf{B}}$ with the standardization functional $\hat{\textbf{V}}$, namely

\[\sqrt{n}(\hat{\textbf{B}} - \boldsymbol{\Lambda}) = \sqrt{n} \left( \hat{\textbf{V}} \frac{1}{n} \sum_{i=1}^n \left[\tilde{\textbf{z}}_i \tilde{\textbf{z}}_i^T  \hat{\textbf{V}}{}^T \hat{\textbf{V}} \tilde{\textbf{z}}_i \tilde{\textbf{z}}_i^T  \right] \hat{\textbf{V}}{}^T - \boldsymbol{\Lambda} \right) ,\]
where $\boldsymbol{\Lambda} = diag(\kappa_1,...,\kappa_p) + (p+2) \textbf{I}_p$ and the inner mean, denoted in the following by $\hat{\textbf{T}}$, converges in probability to the same constant as the matrix $\hat{\textbf{S}}_4$ in the proof of Theorem 8 in \cite{miettinen2014fourth}, namely, to $\boldsymbol{\Lambda}$. We hence have
\[\sqrt{n}(\hat{\textbf{B}} - \boldsymbol{\Lambda}) = \sqrt{n}(\hat{\textbf{V}} - \textbf{I}_p) \boldsymbol{\Lambda} + \boldsymbol{\Lambda} \sqrt{n}(\hat{\textbf{V}}{}^T - \textbf{I}_p) + \sqrt{n}(\hat{\textbf{T}} - \boldsymbol{\Lambda}) + o_P(1), \]
where an arbitrary off-diagonal element of the last term is
\[\frac{1}{n} \sum_{i=1}^n \tilde{z}_{ik} \tilde{\textbf{z}}{}_i^T \sqrt{n} (\hat{\textbf{V}}{}^T \hat{\textbf{V}} - \textbf{I}_p) \tilde{\textbf{z}}{}_i \tilde{z}_{il} + \frac{1}{\sqrt{n}} \sum_{i=1}^n \tilde{z}{}_{ik} \tilde{\textbf{z}}{}_i^T \tilde{\textbf{z}}{}_i \tilde{z}{}_{il}.\]
For the behavior of the latter sum we consult \cite{miettinen2014fourth} and for the first sum, expanding it gives
\[\frac{1}{n} \sum_{i=1}^n \tilde{z}_{ik} \tilde{\textbf{z}}{}_i^T \sqrt{n} (\hat{\textbf{V}}{}^T \hat{\textbf{V}} - \textbf{I}_p) \tilde{\textbf{z}}{}_i \tilde{z}_{il} = \sqrt{n}(\hat{\textbf{V}}{}^T \hat{\textbf{V}} - \textbf{I}_p)_{kl} + \sqrt{n}(\hat{\textbf{V}}{}^T \hat{\textbf{V}} - \textbf{I}_p)_{lk} + o_P(1),\]

where the matrix $\sqrt{n} (\hat{\textbf{V}}{}^T \hat{\textbf{V}} - \textbf{I}_p)$ can be further expanded as $\sqrt{n} (\hat{\textbf{V}}{}^T - \textbf{I}_p) + \sqrt{n}( \hat{\textbf{V}} - \textbf{I}_p) + o_P(1)$. Putting then everything together we have for an off-diagonal element of the FOBI-matrix $\hat{\textbf{B}}$ (and consequently for an off-diagonal element of $\hat{\textbf{C}}{}^4 = \hat{\textbf{B}} - (p+2) \textbf{I}_p$) that
\begin{align*}
\sqrt{n} \hat{b}_{kl} &= \sqrt{n}\hat{q}_{kl} + \sqrt{n}\hat{q}_{lk} + \sum_{m \neq k,l}^p \sqrt{n} \hat{q}_{mkl}\\
 &+ (\kappa_l + p + 4) \sqrt{n} \hat{v}_{kl} + (\kappa_k + p + 4) \sqrt{n} \hat{v}_{lk} + o_P(1).
\end{align*}

In terms of Lemma \ref{lem:app_multi} we then get
\begin{align*}
\sqrt{n} \left(\hat{\textbf{C}}{}^4_{kl} + \delta_{2kl}\hat{v}_{kl}\right) &= \sqrt{n}\hat{q}_{kl} + \sqrt{n}\hat{q}_{lk} + \sum_{m \neq k,l}^p \sqrt{n} \hat{q}_{mkl}\\
 &+ (\kappa_k + p + 4) (\sqrt{n} \hat{v}_{kl} + \sqrt{n} \hat{v}_{lk}) + o_P(1),
\end{align*}
where the effect of the standardization functional $\hat{\textbf{V}}$ vanishes as $\sqrt{n} \hat{v}_{kl} + \sqrt{n} \hat{v}_{lk} = -\sqrt{n} \hat{s}_{kl} + o_P(1)$.

For the corresponding result for the matrix $\hat{\textbf{C}}{}^3$ we first define some notation. Let $\hat{\textbf{V}} \rightarrow_P \textbf{I}_p$ denote an arbitrary standardization matrix and let
\begin{align*}
\hat{\textbf{H}}_k = \hat{\textbf{H}}_k(\hat{\textbf{V}}) &:= \frac{1}{n}\sum_{i=1}^n (\tilde{\textbf{z}}{}_i^T \hat{\textbf{V}}{}^T \textbf{e}_k) \tilde{\textbf{z}}{}_i \tilde{\textbf{z}}{}_i^T = \gamma_k \textbf{E}^{kk} + o_P(1), \\
\hat{N}_{klm} &:= \frac{1}{n}\sum_{i=1}^n \tilde{z}_{ik} \tilde{z}_{il} \tilde{z}_{im} = \delta_{km} \delta_{lm} \gamma_m + o_P(1),
\end{align*}
where $\hat{N}_{klm}$ additionally satisfies $\sqrt{n} \hat{N}_{klm} = (1/\sqrt{n}) \sum_{i=1}^n z_{ik} z_{il} z_{im} - \delta_{kl} \sqrt{n} \bar{z}_m - \delta_{km} \sqrt{n} \bar{z}_l - \delta_{lm} \sqrt{n} \bar{z}_k + o_P(1)$. Using then the above and expanding each of the matrices $\hat{\textbf{V}}$ as $\hat{\textbf{V}} = (\hat{\textbf{V}} - \textbf{I}_p) + \textbf{I}_p$ we write
\begin{align*}
\hat{\textbf{C}}{}^3 &= (\hat{\textbf{V}} - \textbf{I}_p) \left( \sum_{k=1}^p \hat{\textbf{H}}_k \right) (\hat{\textbf{V}}{}^T - \textbf{I}_p) + (\hat{\textbf{V}} - \textbf{I}_p) \left( \sum_{k=1}^p \hat{\textbf{H}}_k \right)  \\
&+ \left( \sum_{k=1}^p \hat{\textbf{H}}_k \right) (\hat{\textbf{V}}{}^T - \textbf{I}_p) + \sum_{k=1}^p \hat{\textbf{H}}_k.
\end{align*}
Using Slutsky's theorem this further yields
\begin{align*}
\sqrt{n} (\hat{\textbf{C}}{}^3 - \sum_{k=1}^p \gamma_k \textbf{E}^{kk}) & =  \sum_{k=1}^p \sqrt{n} (\hat{\textbf{V}} - \textbf{I}_p) \gamma_k \textbf{E}^{kk} + \sum_{k=1}^p \gamma_k \textbf{E}^{kk} \sqrt{n} (\hat{\textbf{V}}{}^T - \textbf{I}_p) \\
&+  \sum_{k=1}^p \sqrt{n} (\hat{\textbf{H}}_k - \gamma_k \textbf{E}^{kk})  + o_P(1).
\end{align*}
Inspecting the result element-wise and using for the last sum the expansion $(\hat{\textbf{H}}_m)_{kl} = \sum_{u=1}^p \hat{v}_{mu} \hat{N}_{klu}$, it easily follows that the element $(k,l)$ of $\sqrt{n}\hat{\textbf{C}}{}^3, \, k \neq l$, satisfies
\[\sqrt{n} \hat{\textbf{C}}{}^3_{kl} = \gamma_l \sqrt{n} \hat{v}_{kl} + \gamma_k \sqrt{n} \hat{v}_{lk} + \sum_{m=1}^p \sum_{u=1}^p \sqrt{n} \hat{v}_{mu} \hat{N}_{klu} + o_P(1). \]

The term consisting of the double sum can further be expanded as
\[\sum_{m,u}\sqrt{n}(\hat{v}_{mu}-\delta_{mu})\delta_{ku} \delta_{lu} \gamma_u + \sum_m \sqrt{n} \hat{N}_{klm} + o_P(1),\]
the first sum of which vanishes as $k \neq l$, leaving only the second sum, which after simplifying has the form
\[\sqrt{n} \hat{r}_{kl} + \sqrt{n} \hat{r}_{lk} + \sum_{m \neq k,l} \sqrt{n} \hat{r}_{mkl}. \]
We then have in terms of Lemma \ref{lem:app_multi}
\[\sqrt{n}\left(\hat{\textbf{C}}{}^3_{kl} + \delta_{1kl}\hat{v}_{kl}\right) = \sqrt{n}\hat{r}_{kl} + \sqrt{n}\hat{r}_{lk} + \sum_{m \neq k,l} \sqrt{n} \hat{r}_{mkl} + \gamma_k (\sqrt{n} \hat{v}_{kl} + \sqrt{n} \hat{v}_{lk}) + o_P(1), \]
where the effect of $\hat{\textbf{V}}$ again vanishes giving then the desired result.

Note, that in the proof we made no assumption whatsoever on the origin of the orthogonal matrix $\hat{\textbf{U}}{}^*$ and thus any choice of IC functional in the standardization leads to the same asymptotic behavior for the estimate $\hat{\textbf{W}}$.
\end{proof}

\begin{proof}[Proof of Theorem \ref{lem:jade_ae}]
As with the affine equivariance of the compound cumulant method in Definition \ref{def:fobi_func}, we again carry out the proof by showing that the optimization problem in Definition \ref{def:jade_func} is invariant under mappings $\textbf{x}_{st} \mapsto \textbf{V}\textbf{x}_{st}$, where $\textbf{V} \in \mathcal{U}$.

We first divide the objective function in two parts
\[D(\textbf{U}, \textbf{x}_{st}) = \alpha D_3(\textbf{U}, \textbf{x}_{st}) + (1 - \alpha) D_4(\textbf{U}, \textbf{x}_{st}),\]
where $D_3(\textbf{U}, \textbf{x}_{st}) = \sum_{i=1}^p \| diag(\textbf{U} \textbf{C}^{3i}(\textbf{x}_{st}) \textbf{U}^T)  \|^2$ denotes the part based on third cumulants and $D_4(\textbf{U}, \textbf{x}_{st}) =  \sum_{i=1}^p \sum_{j=1}^p \| diag(\textbf{U} \textbf{C}^{4ij}(\textbf{x}_{st}) \textbf{U}^T) \|^2$ respectively the part based on fourth cumulants. From the proof of Theorem 9 in \cite{miettinen2014fourth} we have that $D_4(\textbf{U}, \textbf{V} \textbf{x}_{st}) = D_4(\textbf{UV}, \textbf{x}_{st})$ and to complete the proof we thus need the analogical result for $D_3$.

From the proof of Theorem \ref{theo:diag} we see that
\begin{align*}
\textbf{C}^{3i}(\textbf{V}\textbf{x}_{st}) &= \sum_{k=1}^p v_{ik} \textbf{V} E \left[ (\textbf{x}_{st}^T \textbf{e}_k) \textbf{x}_{st} \textbf{x}_{st}^T \right] \textbf{V}^T \\
&= \sum_{k=1}^p v_{ik} \textbf{V} \textbf{C}^{3k}(\textbf{x}_{st}) \textbf{V}^T.
\end{align*}
Denoting $\textbf{W} = (\textbf{w}_1,...,\textbf{w}_p)^T := \textbf{UV}$ and substituting into $D_3(\textbf{U}, \textbf{V}\textbf{x}_{st})$ we then have
\begin{align*}
D_3(\textbf{U}, \textbf{V}\textbf{x}_{st}) &= \sum_{i=1}^p \sum_{d=1}^p \left( \textbf{u}_d^T \left(\sum_{k=1}^p v_{ik} \textbf{V} \textbf{C}^{3k}(\textbf{x}_{st}) \textbf{V}^T \right) \textbf{u}_d \right)^2 \\
&= \sum_{i=1}^p \sum_{d=1}^p \sum_{k=1}^p \sum_{k'=1}^p v_{ik} v_{ik'} \textbf{w}_d^T \textbf{C}^{3k}(\textbf{x}_{st}) \textbf{w}_d \textbf{w}_d^T \textbf{C}^{3k'}(\textbf{x}_{st}) \textbf{w}_d \\
&= \sum_{d=1}^p \sum_{k=1}^p \sum_{k'=1}^p \textbf{w}_d^T \textbf{C}^{3k}(\textbf{x}_{st}) \textbf{w}_d \textbf{w}_d^T \textbf{C}^{3k'}(\textbf{x}_{st}) \textbf{w}_d \left( \sum_{i=1}^p  v_{ik} v_{ik'} \right) \\
&= \sum_{d=1}^p \sum_{k=1}^p \left( \textbf{w}_d^T \textbf{C}^{3k}(\textbf{x}_{st}) \textbf{w}_d \right)^2 \\
&= D_3(\textbf{UV}, \textbf{x}_{st}).
\end{align*}
Combining this with the result for $D_4$ we have thus shown that $D(\textbf{U}, \textbf{V}\textbf{x}_{st}) = D(\textbf{UV}, \textbf{x}_{st})$.
\end{proof}

\begin{proof}[Proof of Theorem \ref{theo:jade_asymp}]
For the consistency of the estimator $\hat{\textbf{W}}$, see the proof of Theorem \ref{theo:fobi_asymp}.

The asymptotic behavior of diagonal elements is covered by Lemma \ref{lem:app_diag} and for the off-diagonal elements we use Lemma \ref{lem:app_multi} which, noting that $\textbf{C}^{3i}(\textbf{z}) = \gamma_i \textbf{E}^{ii}$ and $\textbf{C}^{4ij}(\textbf{z}) = \delta_{ij} \kappa_i \textbf{E}^{ii}$, in conjunction with Lemma \ref{lem:app_diag} now gives
\[\sqrt{n} \hat{\textbf{w}}_{kl} = \frac{\alpha M_3  + (1 - \alpha) M_4}{\alpha (\gamma_k^2 + \gamma_l^2) + (1 - \alpha) (\kappa_k^2 + \kappa_l^2)} - \frac{1}{2}\sqrt{n}\hat{s}_{kl} + o_P(1),\]
where $M_3 = \gamma_k \sqrt{n} \hat{\textbf{C}}{}^{3k}_{kl} - \gamma_l \sqrt{n} \hat{\textbf{C}}{}^{3l}_{kl}$ and $M_4 = \kappa_k \sqrt{n} \hat{\textbf{C}}{}^{4kk}_{kl} - \kappa_l \sqrt{n} \hat{\textbf{C}}{}^{4ll}_{kl}$. Notice again, that as in the proof of Theorem \ref{theo:fobi_asymp}, we again apply Lemma \ref{lem:app_multi} to matrices scaled by the square roots of the weights. We obtain the behavior of fourth cumulants from the proof of theorem in \cite{miettinen2014fourth}.
\begin{align*}
M_4 -\frac{1}{2}(\kappa_k^2 + \kappa_l^2) \sqrt{n} \hat{s}_{kl} = \kappa_k \sqrt{n} \hat{q}_{kl} - \kappa_l \sqrt{n} \hat{q}_{lk} - (\kappa \beta_k - 3 \kappa_l) \sqrt{n} \hat{s}_{kl} + o_P(1).
\end{align*}

To derive the counterpart for third cumulants we again denote the standardization matrix $\hat{\textbf{S}}{}^{1/2}$ by $\hat{\textbf{V}} \rightarrow_P \textbf{I}_p$. With a technique similar to the one used for matrix $\hat{\textbf{C}}{}^3$ in the proof of Theorem \ref{theo:fobi_asymp} we get
\begin{align*}
\sqrt{n} (\hat{\textbf{C}}{}^{3k} - \gamma_k \textbf{E}^{kk}) & = \sqrt{n} (\hat{\textbf{V}} - \textbf{I}_p) \gamma_k \textbf{E}^{kk} + \gamma_k \textbf{E}^{kk} \sqrt{n} (\hat{\textbf{V}}{}^T - \textbf{I}_p)\\
 &+  \sqrt{n} (\hat{\textbf{H}}_k - \gamma_k \textbf{E}^{kk})  + o_P(1).
\end{align*}
Inpsecting the equation element-wise and again using the fact that $(\hat{\textbf{H}}_m)_{kl} = \sum_{u=1}^p \hat{v}_{mu} \hat{N}_{klu}$ (see the proof of Theorem \ref{theo:fobi_asymp}) we then get
\[\sqrt{n}\hat{\textbf{C}}{}^{3k}_{kl} = \sqrt{n}\hat{\textbf{C}}{}^{3k}_{lk} = \gamma_k \sqrt{n} \hat{v}_{lk} + \sqrt{n} \hat{N}_{kkl} + o_P(1), \]
which further yields
\[M_3 -\frac{1}{2}(\gamma_k^2 + \gamma_l^2) \sqrt{n} \hat{s}_{kl} = \gamma_k \sqrt{n} \hat{r}_{kl} - \gamma_l \sqrt{n} \hat{r}_{lk} - \gamma_k^2 \sqrt{n} \hat{s}_{kl} + o_P(1). \]
\end{proof}

%\end{supplement}
\bibliographystyle{plainnatb}
\bibliography{new_references}

%\bibliography{C:/Users/jomivi.UTU/Dropbox/Joni_articles/new_references}

\end{document}